\documentclass[12pt]{amsart}
\usepackage{amsmath,amssymb,mathtools,amsthm,textcomp,mathscinet}
\usepackage{amsfonts,graphicx,enumerate,subcaption,placeins}
\usepackage{color,tikz,float}
\usepackage[mathscr]{eucal}
\usepackage[foot]{amsaddr}
\theoremstyle{definition}
\usepackage[american]{babel}

\usepackage{csquotes}
\usepackage{placeins}

\def \C{{{\rm I{\!\!\!}\rm C}}}

\numberwithin{equation}{section}

\newtheorem{theorem}{\bf Theorem}[section]
\newtheorem{remark}{\bf Remark}[section]

\newtheorem{lemma}{Lemma}[section]
\newtheorem{corollary}{Corollary}[section]
\newtheorem{example}{Example}[section]
\newtheorem{definition}{Definition}[section]
\newtheorem{algorithm}{Algorithm}[section]
\usepackage{chngcntr}
\counterwithout{algorithm}{section}

\newtheoremstyle
    {remarkstyle}
    {}
    {11pt}
    {}
    {}
    {\bfseries}
    {:}
    {     }
    {\thmname{#1} \thmnumber{#2} }

\theoremstyle{remarkstyle}

\def\E{{\mathbb E}}

\begin{document}
\title{Some Time-changed fractional Poisson processes}
\author{\small A. Maheshwari$^*$ and P. Vellaisamy$^*$
}
\address{\small $^*$Postal address: Department of Mathematics,
Indian Institute of Technology Bombay, Powai, Mumbai 400076, INDIA.}
\email{aditya@math.iitb.ac.in, pv@math.iitb.ac.in}
\begin{abstract}
In this paper, we study the fractional Poisson process (FPP) time-changed by an independent L\'evy subordinator and the inverse of the L\'evy subordinator, which we call  TCFPP-I and  TCFPP-II, respectively. Various distributional properties of these processes are established. We show that, under certain conditions, the TCFPP-I has the long-range dependence  property and also its law of iterated logarithm is proved.  It is shown that the TCFPP-II is a renewal process and its waiting time distribution is identified. Its bivariate distributions  and also the governing difference-differential equation are derived. Some specific examples for both the processes are
discussed. Finally,  we present the simulations of the sample paths of these processes.
\end{abstract}

\maketitle
\section{Introduction}
Recently, there has been a considerable interest in studying the fractional Poisson process (FPP) $\{N_\beta(t,\lambda)\}_{t\geq 0}$. The early development of the FPP is due \cite{repin,jumarie,lask}. Later, a rich growth of the literature is contributed by \cite{mainardi04,mnv,beghinejp2009,beghin10}. It is proved in \cite{mnv} that the FPP can be seen as the subordination of the Poisson process by an independent inverse $\beta$-stable subordinator, that is,
\begin{equation}\label{FPP-main}
N_{\beta}(t,\lambda)=N(E_\beta(t),\lambda),~t\geq 0,~ 0<\beta<1,
\end{equation}
where $\{N(t,\lambda)\}_{t\geq0}$ is the Poisson process with rate $\lambda>0$ and $\{E_{\beta}(t)\}_{t\geq0}$ is the inverse $\beta$-stable subordinator. The relation between the inverse $\beta$-stable subordinator and the $\beta$-stable subordinator $\{D_\beta(t)\}_{t\geq 0}$ is 
\begin{equation*}
E_{\beta}(t)=\inf\{r\geq 0:D_{\beta}(r)>t\}, ~~~t\geq0,
\end{equation*}
where the Laplace transform (LT) of the $\beta$-stable subordinator is given by $\mathbb{E}[e^{-sD_{\beta}(t)}]=e^{-ts^\beta}$, $s>0$. \cite{Kumar-TCPP} studied the time-changed Poisson process subordinated with the inverse Gaussian, the first-exit time of the inverse Gaussian, the stable and the tempered stable subordinator.  \cite{sfpp} studied the Poisson process subordinated by an independent $\beta$-stable subordinator $\{D_\beta(t)\}_{t\geq 0}$, called the space fractional Poisson process. In \cite{OrsToa-Berns},  studied the Poisson process subordinated with independent L\'evy subordinator  and  \cite{fnbpfp} studied the FPP subordinated with an independent gamma subordinator to obtain the fractional negative binomial process (FNBP). Observe that the L\'evy subordinator covers most of the special subordinators (see \cite[Theorem 1.3.15]{appm}) considered in the literature.  \\\\
The goal of the present work is to study the FPP $\{N_{\beta}(t,\lambda)\}_{t\geq0}$ time-changed by an independent L\'evy subordinator (hereafter referred to as the subordinator) $\{D_{f}(t)\}_{t\geq0}$ with LT (see \cite[Section 1.3.2]{appm})
\begin{equation}\label{subordinator-LT}
\mathbb{E}[e^{-s D_{f}(t)}]=e^{-tf(s)},
\end{equation}
where 
\begin{equation}\label{Bernstein-function}
f(s)=bs+\int_{0}^{\infty}(1-e^{-s x})\nu(dx),~b\geq0,
\end{equation}
is the Bernstein function. 
Here $b$ is the drift coefficient and $\nu$ is a non-negative L\'evy measure on positive half-line such that 
\begin{equation*}
\int_{0}^{\infty}(x\wedge 1)\nu(dx)<\infty.
\end{equation*}
The assumption $\nu(0,\infty)=\infty$ guarantees that the sample paths of $D_{f}(t)$ are almost surely $(a.s.)$  strictly increasing.  \cite{OrsToa-Berns} studied the process $\{N(D_f(t),\lambda)\}_{t\geq0},$ where $\{D_f(t)\}_{t\geq0}$ is the subordinator with drift coefficient $b=0$.
We investigate the process
\begin{equation*} 
\{Q_{\beta}^{f}(t,\lambda)\}=\{N_{\beta}(D_{f}(t),\lambda)\},~t\geq0,
\end{equation*}
where the time variable $t$ is replaced by an independent subordinator $\{D_f(t)\}_{t\geq0}$ and call the time-changed fractional Poisson process version one (TCFPP-I). The probability mass function ({\it pmf}) of TCFPP-I $ \{Q_{\beta}^{f}(t,\lambda)\}_{t\geq0}$ is obtained and its mean  and covariance functions are computed. We discuss the asymptotic behavior of the covariance function for large $t$. Using these results, we prove the long-range dependence (LRD) property for the TCFPP-I process, under certain conditions. The law of iterated logarithm for the TCFPP-I is also proved. \\
The first-exit time of $\{D_f(t)\}_{t\geq0}$ is defined as
\begin{equation*}
E_{f}(t)=\inf\{r\geq 0:D_{f}(r)>t\}, ~~~t\geq0,
\end{equation*}
which is the right-continuous inverse of the subordinator $\{D_f(t)\}_{t\geq 0}$. We consider also the time-changed fractional Poisson process version two (TCFPP-II) defined  by
\begin{equation*} 
\{W_{\beta}^{f}(t,\lambda)\}=\{N_{\beta}(E_{f}(t),\lambda)\},~t\geq0.
\end{equation*}
The {\it pmf}, mean and covariance functions for the TCFPP-II are derived. We also discuss the asymptotic behavior of the mean and variance functions of the TCFPP-II. The bivariate distribution and the difference-differential equation governing the {\it pmf} of the TCFPP-II are derived. 
Lastly, we present the simulations  for some special TCFPP-I and TCFPP-II processes.\\
The paper is organized as follows. In Section \ref{sec:prelims}, we present some preliminary definitions and results. The TCFPP-I and the TCFPP-II processes are investigated in detail in Section  \ref{sec:tcfpp-1} and \ref{sec:tcfpp-2}, respectively. 
In Section \ref{sec:simulation}, we present the simulations for some specific TCFPP-I and TCFPP-II processes.
\section{Preliminaries} \label{sec:prelims}
\noindent In this section, we present some preliminary results which are required later in the paper. \\
\noindent   The Mittag-Leffler function $L_{\alpha}(z)$ is defined as (see \cite{Mittag-Leffler-original})
\begin{equation}\label{Mittag-Leffler-function}
L_{\alpha}(z)=\sum\limits_{k=0}^{\infty}\frac{z^{k}}{\Gamma(\alpha k+1)},\,\,\,\alpha,z\in \C \text{ and } \Re(\alpha)>0.
\end{equation}
The generalized Mittag-Leffler function $L_{\alpha,\beta}^\gamma(z)$ is defined as (see \cite{Mittag-Leffler-general})
\begin{equation}\label{Mittag-Leffler-general}
L_{\alpha,\beta}^\gamma(z)=\sum_{k=0}^{\infty}\frac{\Gamma(\gamma+k)}{\Gamma(\gamma)\Gamma(\alpha k+\beta)}\frac{z^k}{k!},\,\,\,\alpha,\beta,\gamma,z\in \C \text{ and }\Re(\alpha),\Re(\beta),\Re(\gamma)>0.
\end{equation}

\noindent Let $0<\beta\leq1$. The fractional Poisson process (FPP) $\{N_{\beta}(t,\lambda)\}_{t\geq0}$, which is a generalization of the Poisson process $\{N(t,\lambda)\}_{t\geq0}$, is defined to be a stochastic process for which $p_{_\beta}(n|t,\lambda)=\mathbb{P}[N_{\beta}(t,\lambda)=n]$ satisfies (see \cite{lask,mainardi04,mnv})
\begin{flalign}
&&\partial^{\beta}_{t}p_{_{\beta}}(n|t,\lambda) &= -\lambda\left[ p_{_{\beta}}(n|t,\lambda)-p_{_{\beta}}(n-1|t,\lambda)\right],\,\,\,\text{for } n\geq1,& \label{fpp-definition}\\
&&\partial^{\beta}_{t}p_{_{\beta}}(0|t,\lambda) &= -\lambda p_{_{\beta}}(0|t,\lambda)\nonumber,&
\end{flalign}
$\text{with  }p_{_{\beta}}(n|0,\lambda)=1\text{ if }n=0 \text{ and is zero if }n\geq1.$ Here, $\partial^{\beta}_{t}$ denotes the Caputo-fractional derivative defined by
\begin{equation*}
\partial_{t}^{\beta}f(t)= \begin{cases} 
\hfill \dfrac{1}{\Gamma(1-\beta)}\displaystyle\int_{0}^{t}\dfrac{f'(s)}{(t-s)^{\beta}}ds, \hfill    &0<\beta<1 ,\vspace*{0.3cm} \\
\dfrac{d}{dt}f(t), \,\,\,\,\,\,\,\,\,\,\, \beta=1.&
\end{cases}
\end{equation*}
\noindent The {\it pmf} $p_{_{\beta}}(n|t,\lambda)$ of the FPP is given by (see \cite{lask,mnv}) 
\begin{equation}\label{fppd}
p_{_{\beta}}(n|t,\lambda)=\frac{(\lambda t^{\beta})^n}{n!}\sum_{k=0}^{\infty}\frac{(n+k)!}{k!}\frac{(-\lambda
	t^{\beta})^k}{\Gamma(\beta(k+n)+1)}~.
\end{equation}

\noindent  The mean,  variance  and covariance functions (see \cite{lask,LRD2014}) of the FPP are given by 
\begin{align}
\mathbb{E}[N_{\beta}(t,\lambda)] &= qt^{\beta};~   \mbox{Var}[N_{\beta}(t,\lambda)]=q t^{\beta}+Rt^{2\beta}, \label{fppmean} \\
\text{Cov}[N_{\beta}(s,\lambda),N_{\beta}(t,\lambda)]&=qs^{\beta}+ ds^{2\beta}+ q^{2}[\beta t^{2\beta}B(\beta,1+\beta;s/t)-(st)^{\beta}],~0<s\leq t,\label{fpp-cov}
\end{align}
\noindent
where $q=\lambda/\Gamma(1+\beta)$, $R=\frac{\lambda ^{2}}{\beta}\left(\frac{1}{\Gamma(2\beta)}-\frac{1}{\beta\Gamma^{2}(\beta)}\right)>0$,  $d=\beta q^{2}B(\beta, 1+\beta)$, and $B(a,b;x)=\int_{0}^{x}t^{a-1}(1-t)^{b-1}dt,~0<x<1$, is the incomplete beta function.
\section{Time-changed fractional Poisson process-I}\label{sec:tcfpp-1}
\noindent In this section, we consider the FPP time-changed by a subordinator $\{D_f(t)\}_{t\geq0}$, defined in \eqref{subordinator-LT}, for which the moments $\mathbb{E}[D_f^\rho(t)]<\infty$ for all $\rho>0$. 
Note $\{D_f(t)\}_{t\geq0}$ is an increasing process with $D_f(0)=0~a.s.$
\begin{definition}[\bf TCFPP-I]
	The time-changed fractional Poisson process version one (TCFPP-I) is defined as 
	\begin{equation*}
	\{Q_{\beta}^{f}(t,\lambda)\}=\{N_{\beta}(D_{f}(t),\lambda)\},~t\geq0,
	\end{equation*}
	where $\{N_{\beta}(t,\lambda)\}_{t\geq 0}$ is the FPP and is independent of the subordinator $\{D_{f}(t)\}_{t\geq0}$.
\end{definition}
\noindent We suppress the parameter $\lambda$, unless the context requires, associated with the processes $\{N_\beta(t,\lambda)\}_{t\geq0}$ and $\{Q_\beta^f(t,\lambda)\}_{t\geq0}$. 
\begin{theorem}
	The one-dimensional distributions of the TCFPP-I is given by
	\begin{align}\label{pmf-TCFPP-I}
	\mathbb{P}[Q_{\beta}^{f}(t)=n]=\frac{\lambda^n}{n!}\sum_{k=0}^{\infty}\frac{(n+k)!}{k!}\frac{(-\lambda)^k}{\Gamma(\beta(k+n)+1)}\mathbb{E}[D_{f}^{\beta(n+k)}(t)],~n\geq0.
	\end{align}
	\begin{proof}
		\noindent Let $g_{f}(x,t)$ be the {\it pdf} of $D_{f}(t)$. Then, from \eqref{fppd},
		\begin{align*}
		\mathbb{P}[Q_{\beta}^{f}(t)=n]&=\delta_{\beta}^{f}(n|t,\lambda)=\int_{0}^{\infty}p_{\beta}(n|y,\lambda)g_{f}(y,t)dy\\
		&=\int_{0}^{\infty}\frac{(\lambda y^{\beta})^n}{n!}\sum_{k=0}^{\infty}\frac{(n+k)!}{k!}\frac{(-\lambda
			y^{\beta})^k}{\Gamma(\beta(k+n)+1)}g_{f}(y,t)dy\\
		&=\frac{\lambda^n}{n!}\sum_{k=0}^{\infty}\frac{(n+k)!}{k!}\frac{(-\lambda)^k}{\Gamma(\beta(k+n)+1)}\mathbb{E}[D_{f}^{\beta(n+k)}(t)].
		\qedhere
		\end{align*}
	\end{proof}
\end{theorem}
\begin{remark}It can be seen that the {\it pmf } $\delta_{\beta}^{f}(n|t,\lambda)$ satisfies the normalizing condition  $\sum_{n=0}^{\infty}\delta_{\beta}^{f}(n|t,\lambda)=1$. We have 
	\begin{align*}
	\sum_{n=0}^{\infty}\delta_{\beta}^{f}(n|t,\lambda)&=\sum_{n=0}^{\infty}\frac{\lambda^n}{n!}\sum_{k=0}^{\infty}\frac{(n+k)!}{k!}\frac{(-\lambda)^k}{\Gamma(\beta(k+n)+1)}\mathbb{E}[D_{f}^{\beta(n+k)}(t)]\\
	&=\sum_{n=0}^{\infty}\frac{\lambda^n}{n!}\sum_{r=n}^{\infty}\frac{r!}{(r-n)!}\frac{(-\lambda)^{r-n}}{\Gamma(\beta r+1)}\mathbb{E}[D_{f}^{\beta r}(t)]~~~(\text{taking }r=k+n)\\
	&=\sum_{r=0}^{\infty}\sum_{n=0}^{r}\frac{\lambda^n}{n!}\frac{r!}{(r-n)!}\frac{(-\lambda)^{r-n}}{\Gamma(\beta r+1)}\mathbb{E}[D_{f}^{\beta r}(t)]\\
	&=\sum_{r=0}^{\infty}\frac{\lambda^r\mathbb{E}[D_{f}^{\beta r}(t)]}{\Gamma(\beta r+1)}\sum_{n=0}^{r} \binom{r}{n}(-1)^{r-n}=\sum_{r=0}^{\infty}\frac{\lambda^r\mathbb{E}[D_{f}^{\beta r}(t)]}{\Gamma(\beta r+1)}(1-1)^r\\
	&=\frac{\mathbb{E}[1]}{\Gamma(1)}=1.
	\end{align*}
\end{remark}
\noindent We next present some examples of the TCFPP-I processes. 
\begin{example}[Fractional negative binomial process] Let $\{Y(t)\}_{t\geq0}$ be the gamma subordinator, where $Y(t)\sim G(\alpha,pt)$, the gamma distribution with density 
	\begin{equation*}
	f(x|\alpha,pt)=\frac{\alpha^{pt}}{\Gamma(pt)}e^{-\alpha x}x^{pt-1},~x>0,
	\end{equation*}
	where both $\alpha$ and $p$ are positive. It is known that (see \cite[p. 54]{appm})
	\begin{equation*}
	\mathbb{E}[e^{-sY(t)}]=\left(1+\frac{s}{\alpha}\right)^{-pt}=\exp\left(-pt\log\left(1+s/\alpha\right)\right).
	\end{equation*}
	The fractional negative binomial process (FNBP), introduced and studied in detail in \cite{fnbpfp}, is defined by time-changing the FPP by an independent gamma subordinator, that is, 
	\begin{equation*}
	\{Q^{(1)}_{\beta}(t)\}=\{N_{\beta}(Y(t))\},~t\geq0.
	\end{equation*} 
	It is known that (see \cite[eq. (4.4)]{fnbpfp})
	\begin{equation*}
	\mathbb{E}[Y^\rho(t)]=\frac{\Gamma(pt+\rho)}{\alpha^\rho\Gamma(pt)},~~\rho>0.
	\end{equation*}
	From \eqref{pmf-TCFPP-I}, the {\it pmf} of $Q^{(1)}_\beta(t)$ is
	\begin{equation*}
	\mathbb{P}[Q^{(1)}_\beta(t)=n]=\bigg(\frac{\lambda}{\alpha^{\beta}}\bigg)^{n}\frac{1}{n!\Gamma(pt)}\sum\limits_{k=0}^{\infty}\frac{(n+k)!}{k!}\frac{\Gamma{((n+k)\beta+pt)}}{\Gamma{(\beta(n+k)+1)}}\bigg(\frac{-\lambda}{\alpha^{\beta}}\bigg)^{k},~n\geq0,
	\end{equation*}
	which coincides with the {\it pmf} of the FNBP obtained in \cite{fnbpfp}. Also, it holds that\\ $\sum_{n=0}^{\infty}\mathbb{P}[Q^{(1)}_\beta(t)=n]=1.$
\end{example}
\begin{example}[FPP subordinated by tempered $\alpha$-stable subordinator] Let $\{D_{\alpha}^\mu(t)\}_{t\geq0},\mu>0,0<\alpha<1$ be the tempered $\alpha$-stable subordinator with LT 
	\begin{equation*}
	\mathbb{E}[e^{-sD^\mu_\alpha(t)}]=e^{-t\left((\mu+s)^\alpha-\mu^\alpha\right)}.
	\end{equation*}
	The {\it pdf} of the tempered $\alpha$-stable subordinator is given by (see \cite[eq. (2.2)]{ITS-density})
	\begin{equation*}
	g_\mu(x,t)=e^{-\mu x+\mu^\beta t}g(x,t),~x>0,
	\end{equation*}
	where $g(x,t)$ is the {\it pdf} of the $\alpha$-stable subordinator $\{D_\alpha(t)\}_{t\geq 0}$. 
	The FPP time-changed by an independent tempered $\alpha$-stable subordinator is defined as $$\{Q^{(2)}_{\beta}(t)\}=\{N_\beta(D^\mu_\alpha(t))\},~{t\geq0}.$$
	In this case, the {\it pmf} \eqref{pmf-TCFPP-I} reduces to
	\begin{equation*}
	\mathbb{P}[Q_{\beta}^{(2)}(t)=n]=\frac{\lambda^ne^{\mu^\beta t}}{n!}\sum_{k=0}^{\infty}\frac{(n+k)!}{k!}\frac{(-\lambda)^k}{\Gamma(\beta(k+n)+1)}\mathbb{E}[(D_{\alpha}(t))^{\beta(n+k)}e^{-\mu D_\alpha(t)}],~n\geq0.
	\end{equation*}
	It is easy to show that $\sum_{n=0}^{\infty}\mathbb{P}[Q_{\beta}^{(2)}(t)=n]=1.$
\end{example}
\begin{example}[FPP subordinated by inverse Gaussian subordinator] Let $\{G(t)\}_{t\geq0}$ be the inverse Gaussian subordinator with LT (see \cite[Example 1.3.21]{appm})
	\begin{equation*}
	\mathbb{E}[e^{-sG(t)}]=e^{-t\left(\delta(\sqrt{2s+\gamma^2}-\gamma)\right)},~~\delta,\gamma>0.
	\end{equation*}
	The FPP time-changed by an independent inverse Gaussian subordinator is defined as
	$$\{Q^{(3)}_{\beta}(t)\}=\{N_\beta(G(t))\},~{t\geq0}.$$
	It is known that (see \cite{Jorgenson1982,Kumar-TCPP}) the moments of $\{G(t)\}_{t\geq0}$ are given by
	\begin{equation}\label{IG-moments}
	\mathbb{E}[G^q(t)]=\sqrt{\frac{2}{\pi}}\delta\left(\frac{\delta}{\gamma}\right)^{q-1/2}t^{q+1/2}e^{\delta\gamma t}K_{q-1/2}(\delta\gamma t), ~~\delta,\gamma>0,t\geq0,~q\in\mathbb{R},
	\end{equation}
	where $K_\nu(z),z>0$ is the modified Bessel function of third kind with index $\nu\in\mathbb{R}$. We can substitute the moments of $\{G(t)\}_{t\geq0}$ in \eqref{pmf-TCFPP-I} to obtain the {\it pmf} of $\{Q^{(3)}_\beta(t)\}_{t\geq0}$. Moreover, it can be shown that $\sum_{n=0}^{\infty}\mathbb{P}[Q^{(3)}_\beta(t)=n]=1.$
\end{example}
\noindent We next obtain the mean, the variance and the covariance functions of the TCFPP-I.
\begin{theorem}\label{mean-var-cov-tcfpp-1}
	Let $0<s\leq t<\infty,$~$q=\lambda/\Gamma(1+\beta)$ and $d=\beta q^{2}B(\beta, 1+\beta)$. The distributional properties of the TCFPP-I $\{Q_{\beta}^{f}(t,\lambda)\}_{t\geq0}$ are as follows:\\
	(i) $~~\mathbb{E}[Q^{f}_{\beta}(t)]=q\mathbb{E}[D_{f}^{\beta}(t)]$,\\
	(ii)$~~\text{Var}[Q^{f}_{\beta}(t)]=q\mathbb{E}[D_{f}^{\beta}(t)]\left(1-q\mathbb{E}[D_{f}^{\beta}(t)]\right)+2d\mathbb{E}[D_{f}^{2\beta}(t)]$,
	\begin{flalign}
	\text{(iii) }~\text{Cov}[Q^{f}_{\beta}(s),Q^{f}_{\beta}(t)]&=q\mathbb{E}[D_{f}^{\beta}(s)]+ d\mathbb{E}[D_{f}^{2\beta}(s)]-q^{2}\mathbb{E}[D_{f}^{\beta}(s)]\mathbb{E}[D_{f}^{\beta}(t)] &&\nonumber&\\&~~~~~~~~
	+ q^{2}\beta  \mathbb{E}\left[D_{f}^{2\beta}(t)B\left(\beta,1+\beta;\frac{D_{f}(s)}{D_{f}(t)}\right)\right].\nonumber&
	\end{flalign}
\end{theorem}
\begin{proof}
	Using \eqref{fppmean}, we get 
	\begin{align}\label{tcfpp-mean}
	\mathbb{E}[Q_{\beta}^{f}(t)]&=\mathbb{E}\big[\mathbb{E}[N_{\beta}(D_{f}(t))|D_{f}(t)]\big]=q\mathbb{E}[D_{f}^{\beta}(t)],
	\end{align}
	which proves Part (i). From \eqref{fpp-cov} and \eqref{fppmean}, 
	\begin{equation*}
	\mathbb{E}[N_{\beta}(s)N_{\beta}(t)]=qs^{\beta}+ds^{2\beta}+ q^{2}\beta\left[t^{2\beta}B(\beta,1+\beta;s/t)\right],
	\end{equation*}
	which leads to
	\begin{align}
	\mathbb{E}[Q^{f}_{\beta}(s)Q^{f}_{\beta}(t)]&=\mathbb{E}\left[\mathbb{E}[N_{\beta}(D_{f}(s))N_{\beta}(D_{f}(t))|(D_{f}(s),D_{f}(t))]\right]\nonumber\\
	&=q\mathbb{E}[D_{f}^{\beta}(s)]+ d\mathbb{E}[D_{f}^{2\beta}(s)]+\beta q^{2} \mathbb{E}\left[D_{f}^{2\beta}(t)B\left(\beta,1+\beta;\frac{D_{f}(s)}{D_{f}(t)}\right)\right].\label{bivariate-fnbfp}
	\end{align}
	By  \eqref{tcfpp-mean} and \eqref{bivariate-fnbfp}, Part (iii) follows. Part (ii) follows from Part (iii) by putting $s=t$.
\end{proof}

\noindent{\bf Index of dispersion.} The index of dispersion for a counting process $\{X(t)\}_{t\geq0}$ is defined by (see \cite[p. 72]{cox-lewis})
\begin{equation*}
I(t)=\frac{\text{Var}[X(t)]}{\mathbb{E}[X(t)]}.
\end{equation*}
The stochastic process $\{X(t)\}_{t\geq0}$ is said to be overdispersed if $I(t)>1$ for all $t\geq0$ (see \cite{BegClau14,fnbpfp}). Since the mean of the TCFPP-I $\{Q_{\beta}^{f}(t)\}_{t\geq0}$ is nonnegative, it suffices  to show that Var$[Q_{\beta}^{f}(t)]-\mathbb{E}[Q_{\beta}^{f}(t)]>0$. From Theorem \ref{mean-var-cov-tcfpp-1}, we have that
\begin{align*}
\text{Var}[Q^{f}_{\beta}(t)]-\mathbb{E}[Q^{f}_{\beta}(t)]&=2d\mathbb{E}[D^{2\beta}_{f}(t)]-\big(q\mathbb{E}[D_{f}^{\beta}(t)]\big)^{2}\\
&=\frac{\lambda^{2}}{\beta}\left(\frac{\mathbb{E}[D_{f}^{2\beta}(t)]}{\Gamma(2\beta)}-\frac{(\mathbb{E}[D_{f}^{\beta}(t)])^{2}}{\beta\Gamma^{2}(\beta)}\right)\\
&\geq R\left(\mathbb{E}[D_{f}^{\beta}(t)]\right)^{2},~~(\because \mathbb{E}[D_{f}^{2\beta}(t)]\geq (\mathbb{E}[D_{f}^{\beta}(t)])^{2})
\end{align*}
where $R=\frac{\lambda^2}{\beta}\left(\frac{1}{\Gamma(2\beta)}-\frac{1}{\beta\Gamma^{2}(\beta)}\right)>0$ and $\lambda>0$ for all $\beta\in(0,1)$ (see \cite[Section 3.1]{BegClau14}). Hence, the TCFPP-I exhibits overdispersion.\\

\noindent We next derive the asymptotic expansion for the covariance function of the TCFPP-I process. The following result generalizes Lemma 4.1 of \cite{lrd2016} to the subordinator $\{D_f(t)\}_{t\geq0}$. First, recall the following definition. 
\begin{definition}\label{def-asym}
	Let $f(x)$ and $g(x)$ be positive functions. We say that $f(x)$ is asymptotically equal to $g(x)$, written as $f(x)\sim g(x)$, as $x\to\infty$, if 
	\begin{equation*}
	\lim\limits_{x\rightarrow\infty}\frac{f(x)}{g(x)}=1.
	\end{equation*}
\end{definition}
\begin{theorem}\label{Lemma-asym}
	Let $0<\beta<1$, $0<s \leq t$ and $s$ be fixed. Let $\{D_f(t)\}_{t\geq0}$ be the subordinator with $\mathbb{E}[e^{-sD_f(t)}]=e^{-tf(s)}$, where $f(s)$ is the Bernstein function defined in \eqref{Bernstein-function}. If $\mathbb{E}[D_f^\beta(t)]\rightarrow\infty$ as $t\rightarrow\infty$, then\\
	(i) the asymptotic expansion of $\E\left[D_{f} ^{\beta}(s)D_{f} ^{\beta}(t)\right]$ is 
	\begin{align*}
	\E\left[D_{f} ^{\beta}(s)D_{f} ^{\beta}(t)\right] \sim\E\left[D_{f} ^{\beta}(s)\right]\E\left[D_{f} ^{\beta}(t-s)\right].
	\end{align*}
	(ii) the asymptotic expansion of $\beta \mathbb{E}\left[D_{f} ^{2\beta}(t)B\left(\beta,1+\beta;D_{f} (s)/D_{f} (t)\right)\right]$ is 
	\begin{align*}
	\beta \mathbb{E}\left[D_{f} ^{2\beta}(t)B(\beta,1+\beta;D_{f} (s)/D_{f} (t))\right] \sim\E\left[D_{f} ^{\beta}(s)\right]\E\left[D_{f} ^{\beta}(t-s)\right].
	\end{align*}
	
	\begin{proof} (i) 
		Since the subordinator $\{D_{f} (t)\}_{t\geq0}$ has stationary and independent increments, it suffices to show that
		\begin{equation*}
		\lim\limits_{t\rightarrow\infty}\frac{\E\left[D_{f} ^{\beta}(s)D_{f} ^{\beta}(t)\right]}{\E\left[D_{f} ^{\beta}(s)(D_{f} (t)-D_{f} (s))^\beta\right]}=1.
		\end{equation*}
		Also,  $\{D_{f} (t)\}_{t\geq0}$ is an increasing process with $D_{f} (0)=0~a.s.$ so that
		\begin{align}
		&~~~D_{f} (t)-D_{f} (s)\leq D_{f} (t)~~a.s.\nonumber\\
		\Rightarrow &~~~\E\left[D_{f} ^\beta(s)(D_{f} (t)-D_{f} (s))^{\beta}\right]\leq \E\left[D_{f} ^{\beta}(s)D_{f} ^\beta(t)\right]\nonumber\\
		\Rightarrow &~~~\frac{\E\left[D_{f} ^{\beta}(s)D_{f} ^{\beta}(t)\right]}{\E\left[D_{f} ^{\beta}(s)(D_{f} (t)-D_{f} (s))^\beta\right]}\geq 1.\label{Lemma1oneside}
		\end{align}
		Now consider
		\begin{align}
		\frac{\E\left[D_{f} ^{\beta}(s)D_{f} ^{\beta}(t)\right]}{\E\left[D_{f} ^{\beta}(s)(D_{f} (t)-D_{f} (s))^\beta\right]}\nonumber&=\frac{\E\left[D_{f} ^{\beta}(s)\left\{D_{f} ^{\beta}(t)-(D_{f} (t)-D_{f} (s))^\beta\right\}\right]}{\E\left[D_{f} ^{\beta}(s)(D_{f} (t)-D_{f} (s))^\beta\right]}+1.\\
		\shortintertext{Since $ 0 \leq a^\beta-b^\beta\leq (a-b)^\beta$, for $0<\beta<1$ and $a\geq b\geq 0$,}
		\frac{\E\left[D_{f} ^{\beta}(s)D_{f} ^{\beta}(t)\right]}{\E\left[D_{f} ^{\beta}(s)(D_{f} (t)-D_{f} (s))^\beta\right]}	&\leq \frac{\E\left[D_{f} ^{\beta}(s)\{D_{f} ^{\beta}(t)-(D_{f} ^{\beta}(t)-D_{f} ^\beta(s))\}\right]}{\E\left[D_{f} ^{\beta}(s)D_{f}^\beta (t-s)\right]}+1\nonumber\\
		&=\frac{\E\left[D_{f} ^{2\beta}(s)\right]}{\E\left[D_{f} ^{\beta}(s)D_{f}^\beta (t-s)\right]}+1.\label{Lemma1secondside}
		\end{align}
		From \eqref{Lemma1oneside} and \eqref{Lemma1secondside}, we have that 
		\begin{equation*}
		1\leq \frac{\E\left[D_{f} ^{\beta}(s)D_{f} ^{\beta}(t)\right]}{\E\left[D_{f} ^{\beta}(s)(D_{f} (t)-D_{f} (s))^\beta\right]} \leq \frac{\E\left[D_{f} ^{2\beta}(s)\right]}{\E\left[D_{f} ^{\beta}(s)D^\beta_{f}(t-s)\right]}+1.
		\end{equation*}
		Taking the limit as $t$ tends to infinity in the above equation and using the fact that $\{D_{f} (t)\}_{t\geq0}$ has independent increments, we get
		\begin{align*}
		& 1\leq \lim\limits_{t\rightarrow\infty}\frac{\E\left[D_{f} ^{\beta}(s)D_{f} ^{\beta}(t)\right]}{\E\left[D_{f} ^{\beta}(s)(D_{f} (t)-D_{f} (s))^\beta\right]} \leq 1+\lim\limits_{t\rightarrow\infty}\frac{\E\left[D_{f} ^{2\beta}(s)\right]}{\E\left[D_{f} ^{\beta}(s)\right]\E\left[D_{f} ^\beta(t-s)\right]}\\
		& 1\leq \lim\limits_{t\rightarrow\infty}\frac{\E\left[D_{f} ^{\beta}(s)D_{f} ^{\beta}(t)\right]}{\E\left[D_{f} ^{\beta}(s)(D_{f} (t)-D_{f} (s))^\beta\right]} \leq 1, ~~~ (\text{since }\mathbb{E}[D_f^\beta(t)]\rightarrow\infty\text{ as }t\rightarrow\infty),
		\end{align*}
		which proves   Part (i). \\
		To prove Part (ii), it suffices to show that, in view of Part (i),
		\begin{equation*}
		\lim\limits_{t\rightarrow\infty}\frac{\beta \mathbb{E}\left[D_{f} ^{2\beta}(t)B(\beta,1+\beta;D_{f} (s)/D_{f} (t))\right]}{\E\left[D_{f} ^{\beta}(s)D_{f} ^\beta(t)\right]}=1.
		\end{equation*}
		
		\noindent Note first that 
		\begin{align*}
		B\left(\beta,1+\beta;\frac{D_{f} (s)}{D_{f} (t)}\right)&=\int_{0}^{\frac{D_{f} (s)}{D_{f} (t)}}u^{\beta-1}(1-u)^{\beta}du\nonumber\\
		&\leq \int_{0}^{\frac{D_{f} (s)}{D_{f} (t)}}u^{\beta-1}du~~(\text{since }(1-u)^\beta\leq 1)\nonumber\\
		&=\frac{D_{f} ^\beta(s)}{\beta D_{f} ^\beta(t)},\nonumber
		\end{align*}
		which leads to
		\begin{equation*}
		\beta D^{2\beta}_f(t)	B\left(\beta,1+\beta;\frac{D_{f} (s)}{D_{f} (t)}\right)\leq D^\beta_f(s)D^\beta_f(t).
		\end{equation*}
		Hence,
		\begin{equation}\label{Lemma-it-is-enough-2}
		\frac{\beta \mathbb{E}\left[D_{f} ^{2\beta}(t)B(\beta,1+\beta;D_{f} (s)/D_{f} (t))\right]}{\E\left[D_{f} ^{\beta}(s)D_{f} ^\beta(t)\right]}\leq 1.
		\end{equation}
		On the other hand,
		\begin{align}
		B\left(\beta,1+\beta;D_{f} (s)/D_{f} (t)\right)&=\int_{0}^{\frac{D_{f} (s)}{D_{f} (t)}}u^{\beta-1}(1-u)^{\beta}du\nonumber\\
		&\geq \int_{0}^{\frac{D_{f} (s)}{D_{f} (t)}}u^{\beta-1}(1-u^\beta)du~~(\text{since }(1-u)^\beta\geq 1-u^\beta)\nonumber\\
		&=\frac{1}{\beta}\left(\frac{D_{f} ^\beta(s)}{D_{f} ^\beta(t)}-\frac{D_{f} ^{2\beta}(s)}{2D_{f} ^{2\beta}(t)}\right)\nonumber.
		\end{align}
		This leads to
		\begin{align}
		\frac{\beta \mathbb{E}\left[D_{f} ^{2\beta}(t)B(\beta,1+\beta;D_{f} (s)/D_{f} (t))\right]}{\E\left[D_{f} ^{\beta}(s)D_{f} ^\beta(t)\right]}&\geq \frac{ \mathbb{E}\left[D_{f} ^{2\beta}(t)\left(\frac{D_{f} ^\beta(s)}{D_{f} ^\beta(t)}-\frac{D_{f} ^{2\beta}(s)}{2D_{f} ^{2\beta}(t)}\right)\right]}{\E\left[D_{f} ^{\beta}(s)D_{f} ^\beta(t)\right]}\nonumber\\
		&=\frac{\mathbb{E}\left[D_{f} ^{\beta}(t)D_{f} ^\beta(s)-\dfrac{D_{f} ^{2\beta}(s)}{2}\right]}{\E\left[D_{f} ^{\beta}(s)D_{f} ^\beta(t)\right]}\nonumber\\
		&=1-\frac{\mathbb{E}\left[D_{f} ^{2\beta}(s)\right]}{2\E\left[D_{f} ^{\beta}(s)\right]\E\left[D_{f} ^\beta(t-s)\right]},\label{Lemma-part2-three}
		\end{align}
		using Part (i). By \eqref{Lemma-it-is-enough-2} and \eqref{Lemma-part2-three}, we have that 
		\begin{align*}
		1-&\frac{\mathbb{E}\left[D_{f} ^{2\beta}(s)\right]}{2\E\left[D_{f} ^{\beta}(s)\right]\E\left[D_{f} ^\beta(t-s)\right]}\leq\frac{\beta \mathbb{E}\left[D_{f} ^{2\beta}(t)B(\beta,1+\beta;D_{f} (s)/D_{f} (t))\right]}{\E\left[D_{f} ^{\beta}(s)D_{f} ^\beta(t)\right]}\leq 1.
		\shortintertext{Now taking limit as $t$ tends to infinity in the above inequality, we get}
		1&\leq\lim\limits_{t\to\infty}\frac{\beta \mathbb{E}\left[D_{f} ^{2\beta}(t)B(\beta,1+\beta;D_{f} (s)/D_{f} (t))\right]}{\E\left[D_{f} ^{\beta}(s)D_{f} ^\beta(t)\right]}\leq 1~ (\because \mathbb{E}[D_f^\beta(t)]\rightarrow\infty,\text{ as }t\rightarrow\infty).
		\end{align*}This completes the proof of Part (ii).
	\end{proof}
\end{theorem}
\begin{remark}\label{example-remark}
	The assumption, in Theorem \ref{Lemma-asym}, that $\mathbb{E}[D_f^\beta(t)]\rightarrow\infty$ as $t\rightarrow\infty$,  is satisfied for the following subordinators.\vspace*{0.15cm}\\
	\noindent(a) Gamma subordinator: It is known (see \cite{fnbpfp}) that
	$\mathbb{E}[Y^\beta(t)]=\frac{\Gamma(pt+\beta)}{\alpha^\beta\Gamma(pt)}\sim \left(\frac{pt}{\alpha}\right)^\beta$, which implies $\mathbb{E}[Y^\beta(t)]\to\infty$ as $t\to\infty$.  \vspace*{0.15cm}\\
	(b) Inverse Gaussian subordinator: The moments of inverse Gaussian subordinator $\{G(t)\}_{t\geq0}$ are given in \eqref{IG-moments}.
	The asymptotic expansion of $K_\nu(z)$, for large $z$, is (see \cite[eq. (A.9)]{Jorgenson1982})
	\begin{align}
	K_\nu(z)&=\sqrt{\frac{\pi}{2}}z^{-1/2}e^{-z}\left(1+\tfrac{\mu-1}{8z}+\tfrac{(\mu-1)(\mu-9)}{2!(8z)^2}+\tfrac{(\mu-1)(\mu-9)(\mu-25)}{3!(8z)^3}+\cdots\right)\nonumber\\
	&=\sqrt{\frac{\pi}{2}}z^{-1/2}e^{-z}\left(1+O\left(\tfrac{1}{z}\right)\right),\label{Bessel-asym}
	\end{align}
	where $\mu=4\nu^2$. For $0<\beta<1$, the asymptotic expansion of \eqref{IG-moments}, as $t\to\infty$, is
	\begin{equation*}
	\mathbb{E}[G^\beta(t)]=\left(\frac{\delta t}{\gamma}\right)^\beta\left(1+O\left(\tfrac{1}{t}\right)\right),~~(\text{using }\eqref{Bessel-asym})
	\end{equation*}
	which implies that $\mathbb{E}[G^\beta(t)]\to\infty$ as $t\to\infty.$
\end{remark}
\noindent We next show that the TCFPP-I, under certain conditions on the subordinator $\{D_f(t)\}_{t\geq0}$, possesses the LRD property.  There are various definitions in the literature for the LRD property of a stochastic process. We now present the definition (see \cite{ovi-lrd,lrd2016}) that  will be used in this paper.
\begin{definition}\label{LRD-definition}
	Let $0<s<t$ and $s$ be fixed. Assume a stochastic process $\{X(t)\}_{t\geq0}$ has the correlation function Corr$[X(s),X(t)]$ that satisfies
	\begin{equation*}
	c_1(s)t^{-d}\leq\text{Corr}[X(s),X(t)]\leq c_2(s)t^{-d},
	\end{equation*}
	for large $t$, $d>0$, $c_1(s)>0$ and $c_2(s)>0$. That is, 
	\begin{equation}\label{LRD-defn2}
	\lim\limits_{t\to\infty}\frac{\text{Corr}[X(s),X(t)]}{t^{-d}}=c(s),
	\end{equation}for some $c(s)>0$ and $d>0.$  We say $\{X(t)\}_{t\geq0}$ has the long-range dependence (LRD) property if $d\in(0,1)$  and has the short-range dependence (SRD) property if $d\in(1,2)$.
\end{definition}
\noindent Note \eqref{LRD-defn2} implies that Corr$[X(s),X(t)]$ behaves like $t^{-d}$, for large $t$.

\begin{theorem}\label{LRD-theorem}Let $\{D_{f}(t)\}_{t\geq0}$ be such that  $\mathbb{E}[D_f^\beta(t)]\sim k_1t^\rho$ and $\mathbb{E}[D_f^{2\beta}(t)]\sim k_2t^{2\rho}, $ for some  $0< \rho <1,$ and positive  constants $k_1$ and $k_2$ with $k_2\geq k_1^2$. Then the TCFPP-I $\{Q^f_{\beta}(t)\}_{t\geq 0}$ has the LRD property.
	\begin{proof}\noindent Consider the last term of $\text{Cov}[Q^f_{\beta}(s),Q_{\beta}^f(t)]$ given in Theorem \ref{mean-var-cov-tcfpp-1} (iii), namely,
		\begin{equation*}
		\beta q^{2} \mathbb{E}\left[D_f^{2\beta}(t)B\left(\beta,1+\beta;\frac{D_f(s)}{D_f(t)}\right)\right].
		\end{equation*}
		Using Theorem \ref{Lemma-asym} (ii), we get for large $t$,
		\begin{align}
		q^{2}\beta \mathbb{E}[D_f^{2\beta}(t)B\left(\beta,1+\beta;\frac{D_f(s)}{D_f(t)}\right)
		&\sim q^{2}\E[D_f^{\beta}(s)]\E[D_f^{\beta}(t-s)]\label{autocovariance-last-summand-1}.
		\end{align}
		\noindent Using \eqref{autocovariance-last-summand-1}, Theorem \ref{mean-var-cov-tcfpp-1} (iii) and $\mathbb{E}[D_f^\beta(t)]\sim k_1t^\rho$, we get, 
		\begin{align}
		\text{Cov}[Q^f_{\beta}(s),Q^f_{\beta}(t)]&\sim q\mathbb{E}[D_f^{\beta}(s)]+d\mathbb{E}[D_f^{2\beta}(s)]-q^{2}\mathbb{E}[D_f^{\beta}(s)]k_1t^\rho+q^{2}\mathbb{E}[D_f^{\beta}(s)]k_1(t-s)^\rho\nonumber\\
		&=q\mathbb{E}[D_f^{\beta}(s)]+d\mathbb{E}[D_f^{2\beta}(s)]-q^{2}\mathbb{E}[D_f^{\beta}(s)]k_1(t^\rho-(t-s)^\rho)\nonumber\\
		&\sim q\mathbb{E}[D_f^{\beta}(s)]+d\mathbb{E}[D_f^{2\beta}(s)]-q^{2}k_1s\rho \mathbb{E}[D_f^{\beta}(s)]t^{\rho-1},\label{covariance-large-t}
		\end{align}
		since $t^{\rho}-(t-s)^{\rho} \sim \rho s t^{\rho-1}$ for large $t$, and $\rho>0.$
		
		\noindent Similarly, from Theorem \ref{mean-var-cov-tcfpp-1} (ii) and $\mathbb{E}[D_f^{2\beta}(t)]\sim k_2t^{2\rho}$, we have that
		\begin{align}
		\text{Var}[Q^f_{\beta}(t)]&\sim qk_1t^\rho- q^{2}(k_1t^\rho )^2+2dk_2t^{2\rho}\nonumber\\
		&\sim 2dk_2t^{2\rho}- q^{2}k_1^2t^{2\rho}~~(\text{see Definition }\ref{def-asym})\nonumber\\
		&=  \frac{\lambda^2}{\beta}\left(\frac{k_2}{\Gamma(2\beta)}-\frac{k_1^2}{\beta\Gamma^2(\beta)}\right)t^{2\rho}\nonumber\\
		&=d_1t^{2\rho},\label{variance-large-t}
		\end{align}
		where $d_{1}=\frac{\lambda^2}{\beta}\left(\frac{k_2}{\Gamma(2\beta)}-\frac{k_1^2}{\beta\Gamma^2(\beta)}\right)$. 
		\noindent Thus, from \eqref{covariance-large-t} and \eqref{variance-large-t}, we have for large $t>s$,
		\begin{align}
		\text{Corr}[Q^f_{\beta}(s),Q^f_{\beta}(t)]&\sim\frac{q\mathbb{E}[D_f^{\beta}(s)]+d\mathbb{E}[D_f^{2\beta}(s)]-q^{2}k_1s\rho \mathbb{E}[D_f^{\beta}(s)]t^{\rho-1}}{\sqrt{\text{Var}[Q^f_{\beta}(s)]}\sqrt{d_1t^{2\rho}}}\nonumber\\
		&= \left(\frac{q\mathbb{E}[D_f^{\beta}(s)]+d\mathbb{E}[D_f^{2\beta}(s)]}{\sqrt{d_{1}\text{Var}[Q^f_{\beta}(s)]}}\right)t^{-\rho}-\frac{q^{2}k_1s\rho \mathbb{E}[D_f^{\beta}(s)]}{\sqrt{t^{2\rho}d_{1}}\sqrt{\text{Var}[Q^f_{\beta}(s)]}}t^{-1}\nonumber\\
		&\sim \left(\frac{q\mathbb{E}[D_f^{\beta}(s)]+d\mathbb{E}[D_f^{2\beta}(s)]}{\sqrt{d_{1}\text{Var}[Q^f_{\beta}(s)]}}\right)t^{-\rho}\nonumber,  		
		\end{align}
		\noindent which decays like the power law $t^{-\rho},~0<\rho<1$. Hence, the TCFPP-I exhibits the LRD property.\end{proof}
\end{theorem}
\begin{remark}
	From Remark \ref{example-remark}, it can be seen that moments of the gamma subordinator has the asymptotic expansion  $\mathbb{E}[Y^\beta(t)]\sim (p/\alpha)^\beta t^\beta$ and $\mathbb{E}[Y^{2\beta}(t)]\sim \left(p/\alpha\right)^{2\beta} t^{2\beta}$. Therefore, the FNBP $\{Q^{(1)}_\beta(t)\}_{t\geq0}$ exhibits the LRD property. 
	Similarly, for the inverse Gaussian subordinator $\{G(t)\}_{t\geq0}$, we have  the asymptotic expansion $\mathbb{E}[G^\beta(t)]\sim \left(\delta /\gamma\right)^\beta t^\beta$ and $\mathbb{E}[G^{2\beta}(t)]\sim \left(\delta /\gamma\right)^{2\beta} t^{2\beta}$. Hence, $\{Q^{(3)}_\beta(t)\}_{t\geq0}$ also has the LRD property.
	
\end{remark}
\begin{definition}
	We call a function $l:(0,\infty)\rightarrow(0,\infty)$ regularly varying at 0+ with index $\alpha\in\mathbb{R}$ (see \cite{bertoin}) if 
	$$\lim_{x\rightarrow 0+}\frac{l(\lambda x)}{l(x)}=\lambda^\alpha,~\text{for}~ \lambda>0.$$
\end{definition}
\noindent   We first reproduce the following law of iterated logarithm (LIL) for the subordinator from \cite[Chapter III, Theorem 14]{bertoin}.
\begin{lemma}Let $X(t)$ be a subordinator with $\mathbb{E}[e^{-sX(t)}]=e^{-tf(s)}$, where $f(s)$ is regularly varying at $0+$ with index $\alpha\in(0,1)$. Let $h$ be the inverse function of $f$ and
	\begin{equation*}
	g(t)=\frac{\log\log t}{h(t^{-1}\log\log t)},~(e<t).
	\end{equation*}
	Then
	\begin{equation}\label{LIL-sub}
	\liminf_{t\to\infty}\frac{X(t)}{g(t)}=\alpha(1-\alpha)^{(1-\alpha)/\alpha},~~a.s.
	\end{equation}
\end{lemma}
\noindent We next prove the LIL for the TCFPP-I.
\begin{theorem}[Law of iterated logarithm]
	Let the Laplace exponent $f(s)$ of the subordinator $\{D_{f}(t)\}_{t\geq0}$ be regularly varying at 0+ with index $\alpha\in(0,1)$. Then, for $ 0<\beta< 1$, 
	\begin{equation*}
	\liminf_{t\rightarrow\infty}\frac{Q_{\beta}^{f}(t)}{(g(t))^{\beta}}=\lambda E_{\beta}(1)\alpha^\beta\left(1-\alpha\right)^{\beta(1-\alpha)/\alpha}~~{a.s.},
	\end{equation*}
	where 
	\begin{equation*}
	g(t)=\frac{\log\log t}{f^{-1}(t^{-1}\log\log t)}~(t>e),
	\end{equation*}
	and $E_\beta(t)$ is the inverse $\beta$-stable subordinator.
	\begin{proof}
		Since  $E_{\beta}(t)\stackrel{d}{=}t^{\beta}E_{\beta}(1)$ (see \cite[Proposition 3.1]{MeerSche2004}), we have
		\begin{equation*}
		Q_{\beta}^{f}(t)=N(E_{\beta}(D_{f}(t)))\stackrel{d}{=}N((D_{f}(t))^{\beta}E_{\beta}(1)).
		\end{equation*}
		The law of large numbers for the Poisson process implies
		\begin{equation*}
		\lim_{t\rightarrow\infty}\frac{N(t)}{t}=\lambda,~{a.s.}
		\end{equation*} Note that $D_f(t)\to\infty$, $a.s.$ as $t\to\infty$ (see  \cite[page 73]{bertoin}). Consider now,
		\begin{align*}
		\liminf_{t\rightarrow\infty}\frac{Q_{\beta}^{f}(t)}{(g(t))^{\beta}}&=\liminf_{t\rightarrow\infty}\frac{N(E_{\beta}(D_{f}(t)))}{(g(t))^{\beta}}=\liminf_{t\rightarrow\infty}\frac{N(D^{\beta}_{f}(t)E_{\beta}(1))}{(g(t))^{\beta}}\\
		&=\liminf_{t\rightarrow\infty}\frac{N(D^{\beta}_{f}(t)E_{\beta}(1))}{D^{\beta}_{f}(t)E_{\beta}(1)}\frac{D^{\beta}_{f}(t)E_{\beta}(1)}{(g(t))^{\beta}}\\
		&=\lambda E_{\beta}(1)\liminf_{t\rightarrow\infty}\frac{D^{\beta}_{f}(t)}{(g(t))^{\beta}},~{a.s.}\\
		&=\lambda E_{\beta}(1)\left(\liminf_{t\rightarrow\infty}\frac{D_{f}(t)}{g(t)}\right)^{\beta},~{a.s.}\\
		&=\lambda E_{\beta}(1)\alpha^\beta\left(1-\alpha\right)^{\beta(1-\alpha)/\alpha}~{a.s.},
		\end{align*}
		where the last step follows from \eqref{LIL-sub}.
	\end{proof}
	\noindent When $\beta=1$, the LIL for the time-changed Poisson process $\{Q_1^f(t)\}\stackrel{d}{=}\{N_1(D_f(t))\},~t\geq0,$ (discussed in \cite{OrsToa-Berns}) can be proved in a similar way and is stated below.
	\begin{corollary}\label{LIL-Coro}
		Let the Laplace exponent $f(s)$ of the subordinator $\{D_{f}(t)\}_{t\geq0}$ be regularly varying at 0+ with index $\alpha\in(0,1)$. Then 
		\begin{equation}\label{LIL-tcpp}
		\liminf_{t\rightarrow\infty}\frac{Q_1^{f}(t)}{g(t)}=\lambda \alpha\left(1-\alpha\right)^{(1-\alpha)/\alpha}~~{a.s.},
		\end{equation}
		where 
		\begin{equation*}
		g(t)=\frac{\log\log t}{f^{-1}(t^{-1}\log\log t)}~(t>e).
		\end{equation*}
	\end{corollary}
\end{theorem}
\begin{example}
	The space fractional Poisson process, introduced in \cite{sfpp}, defined by time changing the Poisson process by an independent $\alpha$-stable subordinator, that is,
	\begin{equation*}
	\widetilde{N}_\alpha(t)=N(D_\alpha (t)),~t\geq0,~0<\alpha<1,
	\end{equation*} 
	where $\{D_\alpha(t)\}_{t\geq 0}$ is the $\alpha$-stable subordinator with LT $\mathbb{E}[e^{-sD_\alpha (t)}]=e^{-ts^\alpha}$. Here, the corresponding Bernstein function $f(s)=s^\alpha$ is regularly varying with index $\alpha\in(0,1)$. Therefore, by Corollary \ref{LIL-Coro}, we have the LIL for the space fractional Poisson process with 
	\begin{equation*}
	g(t)=\frac{\log\log t}{(t^{-1}\log\log t)^{1/\alpha}},~(t>e).
	\end{equation*}
\end{example}
%

\section{Time-changed fractional Poisson process-II}\label{sec:tcfpp-2}
\noindent The first-exit time of the subordinator $\{D_{f}(t)\}_{t\geq0}$ is its right-continuous inverse, defined by
\begin{equation*}
E_{f}(t)=\inf\{r\geq 0:D_{f}(r)>t\}, ~~~t\geq0,
\end{equation*}
and is called an {\it inverse subordinator} (see \cite{bertoin}). Note that for any $\rho>0$, $\mathbb{E}[E_f^\rho(t)]<\infty$ (see \cite[Section  2.1]{Inverse-subordinator}). We now consider the FPP time-changed by an inverse subordinator. 
\begin{definition}[\bf TCFPP-II] The time-changed fractional Poisson process version two (TCFPP-II) is defined as 
	\begin{equation*}
	\{W_{\beta}^{f}(t)\}=\{N_{\beta}(E_{f}(t))\},~t\geq0,
	\end{equation*}
	where $\{N_{\beta}(t)\}_{t\geq 0}$ is the FPP and is independent of the inverse subordinator $\{E_{f}(t)\}_{t\geq0}$. 
\end{definition}
\noindent We now present some results and distributional properties of the TCFPP-II. The proofs of some of them are shortened or omitted to avoid repetition from the previous section.  \\\\
\noindent 
The one-dimensional distributions of the TCFPP-II can be written as
\begin{align}\label{pmf-TCFPP-II}
\eta_{\beta}^{f}(n|t,\lambda)=\mathbb{P}[W_{\beta}^{f}(t)=n]=\frac{\lambda^n}{n!}\sum_{k=0}^{\infty}\frac{(n+k)!}{k!}\frac{(-\lambda)^k}{\Gamma(\beta(k+n)+1)}\mathbb{E}[E_{f}^{\beta(n+k)}(t)],
\end{align}
which follows from \eqref{fppd}.	
\begin{theorem}\label{mean-var-cov-tcfpp-2}
	Let $0<s\leq t$. Then\\
	(i) $~~\mathbb{E}[W^{f}_{\beta}(t)]=q\mathbb{E}[E_{f}^{\beta}(t)]$,\\
	(ii)$~~\text{Var}[W^{f}_{\beta}(t)]=q\mathbb{E}[E_{f}^{\beta}(t)]\left(1-q\mathbb{E}[E_{f}^{\beta}(t)]\right)+2d\mathbb{E}[E_{f}^{2\beta}(t)]$,
	\begin{flalign}
	\text{(iii) }~\text{Cov}[W^{f}_{\beta}(s),W^{f}_{\beta}(t)]&=q\mathbb{E}[E_{f}^{\beta}(s)]+ d\mathbb{E}[E_{f}^{2\beta}(s)]-q^{2}\mathbb{E}[E_{f}^{\beta}(s)]\mathbb{E}[E_{f}^{\beta}(t)] &&\nonumber&\\&~~~~~~~
	+ q^{2}\beta  \mathbb{E}\left[E_{f}^{2\beta}(t)B\left(\beta,1+\beta;\frac{E_{f}(s)}{E_{f}(t)}\right)\right].\nonumber&
	\end{flalign}
	\begin{proof}
		The proof is similar to the proof of Theorem \ref{mean-var-cov-tcfpp-1} and hence is omitted.
	\end{proof}
\end{theorem}

\noindent 
We next discuss the asymptotic behavior of moments of the TCFPP-II. The mean and variance functions contain the term of the form $\mathbb{E}[E^\beta_f(t)]$. Therefore, we study the asymptotic behavior of $\mathbb{E}[E^\beta_f(t)]$.
It will be studied using the Tauberian
theorem (see \cite[Theorem 4.1]{Taqqu2010} and \cite[p. 10]{bertoin}), which we reproduce here.  Recall that a function
$\ell(t)$, $t>0$, is \textit{slowly varying} at $0$ (respectively
$\infty$) if for all $c>0$, $\lim(\ell(ct)/\ell(t)) = 1$, as
$t\rightarrow 0$ (respectively $t\rightarrow \infty$). 
\begin{theorem}\label{Tauberian}
	(Tauberian theorem) Let $\ell:(0,\infty) \rightarrow (0,\infty)$ be a slowly varying
	function at $0$ (respectively $\infty$) and let $\rho \geq 0$.  Then
	for a function $U: (0,\infty) \rightarrow (0,\infty)$, the
	following are equivalent:
	
	(i) $U(x) \sim x^\rho \ell(x)/\Gamma(1 + \rho), \quad x \rightarrow 0$
	(respectively $x \rightarrow \infty$).
	
	(ii) $\widetilde{U}(s) \sim s^{-\rho-1} \ell(1/s), \quad s \rightarrow \infty$
	(respectively $s \rightarrow 0$),\\
	where $\widetilde{U}(s)$ is the LT of $U(x)$.
\end{theorem}

\noindent Let $M_p(t)=\E[E_{f}^p(t)],~p>0$. The LT of the $p$-th moment of $E_{f}(t)$ is given by (see \cite{Arun-ITS})
\begin{align}
\widetilde{M}_p(s)=\frac{\Gamma(1+p)}{s(f(s))^p},\label{asympt-moment}
\end{align}
where $f(s)$ is the Bernstein function associated with $\{D_f(t)\}_{t\geq0}$. The asymptotic moments can be specifically computed for special cases, which also serves examples of the TCFPP-II processes.

\begin{example}[FPP subordinated with inverse gamma subordinator] We study the FPP time-changed by the inverse of the gamma subordinator $\{Y(t)\}_{t\geq0}$, with corresponding Bernstein function $f(s)=p\log(1+s/\alpha)$. The right-continuous inverse of the gamma subordinator $\{Y(t)\}_{t\geq0}$  is defined as 
	\begin{equation*}
	E_{Y}(t)=\inf\{r\geq0:Y(r)>t\}, ~t\geq0.
	\end{equation*}
	We study the asymptotic behavior of the mean of $\{W_\beta^{(1)}(t)\}_{t\geq0}=\{N_\beta(E_Y(t))\}_{t\geq0}$, that is, the function $\mathbb{E}[W_\beta^{(1)}(t)]=q\mathbb{E}[E_Y^\beta(t)]$. 
	The LT of $\mathbb{E}[E^\beta_Y(t)]$ is given by
	\begin{equation*}
	\widetilde{M}_\beta(s)=\mathscr{L}\left[\mathbb{E}[E^\beta_Y(t)]\right]=\frac{\Gamma(1+\beta)}{s(p\log(1+s/\alpha))^\beta}.
	\end{equation*} Note that  $p\log(1+s/\alpha)\sim ps/\alpha,$ as $s\to0$. Now using Theorem \ref{Tauberian}, we get (see also \cite[Proposition 4.1]{Arun-inverse-gamma})
	\begin{align*}
	\E[W^{(1)}_\beta(t)]&=q\E[E^\beta_Y(t)]\sim q(t\alpha/p)^\beta,~\text{as }t\to\infty.
	\end{align*}
	The asymptotic behavior of variance function of $\{W^{(1)}_\beta(t)\}$ can also be computed using above expression. 
\end{example}
\begin{example}[FPP subordinated with the inverse tempered $\alpha$-stable subordinator] Consider the FPP subordinated with the inverse tempered $\alpha$-stable subordinator $\{E^\mu_{\alpha}(t)\}_{t\geq0}$. The inverse tempered $\alpha$-stable subordinator is introduced by \cite{Arun-ITS} and they studied its asymptotic behavior of moments. 
	The $p$-th moment of $E^\mu_{\alpha}(t)$ satisfies (see \cite[Proposition 3.1]{Arun-ITS})
	\begin{eqnarray*}
		\E[(E^\mu_\alpha(t))^p] \sim \left\{
		\begin{array}{ll}
			\displaystyle \frac{\Gamma(1+p)}{\Gamma(1+p\alpha)}t^{p\alpha},&~~~\mbox{as}~ t\rightarrow 0,\\&\\
			\dfrac{\lambda^{p(1-\alpha)}}{\alpha^p}t^p,& ~~~\mbox{as}~t\rightarrow\infty.
		\end{array}\right.
	\end{eqnarray*}
	\noindent Therefore, we have that
	\begin{align*}
	\E[W_\beta^{(2)}(t)]=q\E[(E^\mu_{\alpha}(t))^\beta]\sim\left\{\begin{array}{ll}
	\displaystyle \frac{q\Gamma(1+\beta)}{\Gamma(1+\beta\alpha)}t^{\beta\alpha},&~~~\mbox{as}~ t\rightarrow 0,\vspace*{0.3cm}\\
	\dfrac{q\lambda^{\beta(1-\alpha)}}{\alpha^\beta}t^\beta,& ~~~\mbox{as}~t\rightarrow\infty.
	\end{array}\right.
	\end{align*}	
\end{example}
\begin{example}[FPP subordinated with inverse of the inverse Gaussian subordinator]
	The right-continuous inverse of the inverse Gaussian subordinator $\{G(t)\}_{t\geq0}$, with corresponding Bernstein function $f(s)=\delta\left(\sqrt{2s+\gamma^2}-\gamma\right)$, denoted by $\{E_G(t)\}_{t\geq0}$, defined as (see \cite{kumhit})
	\begin{equation*}
	E_G(t)=\inf\{r\geq 0:G(r)>t  \},~t\geq0.
	\end{equation*} Hence from \eqref{asympt-moment} 
	\begin{align*}
	\widetilde{M}_p(s)&=\mathscr{L}\left[\E[E_G^p(t)]\right]=\frac{\Gamma(1+p)}{s\left(\delta\left(\sqrt{2s+\gamma^2}-\gamma\right)\right)^p},
	\shortintertext{where $p>0$. This gives}
	\widetilde{M}_p(s)&\sim \left\{\begin{array}{ll}
	\displaystyle \frac{\Gamma(1+p)}{(\delta/\gamma)^p}s^{-1-p},&~~~\mbox{as}~ s\rightarrow 0,\vspace*{0.3cm}\\ 
	\dfrac{\Gamma(1+p)}{(\delta\sqrt{2})^p}s^{-1-p/2},& ~~~\mbox{as}~s\rightarrow\infty.
	\end{array}\right.
	\end{align*}
	Using above result and Theorem \ref{Tauberian}, we get
	\begin{align*}
	\E[(E_G(t))^p]\sim\left\{\begin{array}{ll}
	\displaystyle \frac{\Gamma(1+p)t^{p/2}}{\Gamma(1+p/2)(\delta\sqrt{2})^p},&~~~\mbox{as}~ t\rightarrow 0,\vspace*{0.3cm}\\
	\left(\dfrac{\gamma}{\delta}\right)^p t^p,& ~~~\mbox{as}~t\rightarrow\infty.
	\end{array}\right.
	\end{align*}	
	We finally get the asymptotic moments as
	\begin{align*}
	\E[W^{(3)}_\beta(t)]=q\E[(E_G(t))^\beta]\sim\left\{\begin{array}{ll}
	\displaystyle \frac{q\Gamma(1+\beta)}{\Gamma(1+\beta/2)(\delta\sqrt{2})^\beta}t^{\beta/2},&~~~\mbox{as}~ t\rightarrow 0,\vspace*{0.3cm}\\
	q\left(\dfrac{\gamma}{\delta}\right)^\beta t^\beta,& ~~~\mbox{as}~t\rightarrow\infty.
	\end{array}\right.
	\end{align*}
\end{example}

\noindent We next show that the TCFPP-II $\{W_{\beta}^{f}(t)\}_{t\geq0}$ is a renewal process. We begin with the following lemma.\\ \noindent Let $\{D_f(t)\}_{t\geq0}$ be a subordinator with the associated Bernstein function $f(s)$. Let $\{E_f(t)\}_{t\geq0}$ be the right-continuous inverse of $\{D_f(t)\}_{t\geq0}$. We call, rather loosely, $E_f(t)$ the inverse subordinator corresponding to $f(s)$.
\begin{lemma}
	Let $\{E_{f_{1}}(t)\}_{t\geq 0}$ and $\{E_{f_{2}}(t)\}_{t\geq 0}$ be two independent inverse subordinators corresponding to Bernstein functions $f_1(s)$ and $f_2(s)$, respectively. Then 
	\begin{equation}\label{G-TCFPP-II}
	\{E_{f_1}(E_{f_2}(t))\}\stackrel{d}{=}\{E_{f_1\circ f_2}(t)\},~t\geq0,
	\end{equation}
	where $(f_1\circ f_2)(s)=f_1(f_2(s)).$
	\begin{proof}
		Consider two independent subordinators $\{D_{f_1}(t)\}_{t\geq0}$ and $\{D_{f_2}(t)\}_{t\geq0}$ with
		\begin{equation*}
		\mathbb{E}[e^{-s D_{f_1}(t)}]=e^{-tf_1(s)}~~~\text{and}~~~\mathbb{E}[e^{-s D_{f_2}(t)}]=e^{-tf_2(s)},
		\end{equation*}
		where $f_1(s)$ and $f_{2}(s)$ are the associated Bernstein functions. We claim that 
		\begin{equation}\label{G-TCFPP-I}
		\{D_{f_2}(D_{f_1}(t))\}\stackrel{d}{=}\{D_{f_1\circ f_2}(t)\},~~~t\geq0,
		\end{equation}
		where $\circ$ denotes the composition of functions. To see this, let us compute the LT of the left-hand side
		\begin{equation*}
		\mathbb{E}\left[e^{-s D_{f_2}(D_{f_1}(t))}\right]=\mathbb{E}\left[\mathbb{E}\left[e^{-s D_{f_2}(D_{f_1}(t))}|D_{f_1}(t)\right]\right]=\mathbb{E}\left[e^{-f_2(s)D_{f_1}(t)}\right]=e^{-tf_1(f_2(s))},~s>0.  
		\end{equation*}
		Since $f(s)=(f_1\circ f_2)(s)$ is again a Bernstein function (see \cite[Remark 5.28 (ii)]{Bernstein-book}) and $\{D_{f_2}(D_{f_1}(t))\}_{t\geq0}$ is a L\'evy process (see \cite[Theorem 1.3.25]{appm}), it follows that $\{D_{f}(t)\}_{t\geq0}$ is a subordinator with associated Bernstein function $f(s)=(f_1\circ f_2)(s)$.

		\noindent Consider next have the inverse subordinators defined by
		\begin{equation*}
		E_{f_1}(t)=\inf\{r\geq 0:D_{f_1}(r)>t\}~~\text{and}~~E_{f_2}(t)=\inf\{r\geq 0:D_{f_2}(r)>t\}, ~t\geq0.
		\end{equation*}
		Then the process
		\begin{align*}
		E_{f_1\circ f_2}(t)&=\inf\{r\geq 0:D_{f_1\circ f_2}(r)>t\}\\
		&=\inf\{r\geq 0:D_{f_2}(D_{f_1}(r))>t\}~~~(\text{using \eqref{G-TCFPP-I}}).
		\shortintertext{By the property of right-continuous inverse, we have that   $\{D_{f_2}(D_{f_1}(r))>t\}=\{E_{f_2}(t)<D_{f_1}(r)\}$, and hence}
		E_{f_1\circ f_2}(t)&=\inf\{r\geq 0:D_{f_1}(r)>E_{f_{2}}(t)\} =E_{f_1}(E_{f_2}(t)),
		\end{align*}
		which completes the proof.
	\end{proof}
\end{lemma}
\begin{corollary}\label{lemma-ebeta}
	Let $\{E_\beta(t)\}_{t\geq 0}$ be inverse $\beta$-stable subordinator corresponding to $f_1(s)=s^\beta$, and $\{E_f(t)\}_{t\geq 0}$ be an inverse subordinator corresponding to $f_2(s)=f(s)$. Then from \eqref{G-TCFPP-II}, 
	\begin{equation}\label{lemma-ebeta-eq}
	\{E_\beta(E_f(t))\}_{t\geq0}\stackrel{d}{=}\{E_{\phi}(t)\}_{t\geq0},
	\end{equation}
	where $\phi(s)=(f(s))^\beta$. 
\end{corollary}
\begin{remark}
	One can  further generalize the TCFPP-I process $\{Q_{\beta}^{f}(t)\}_{t\geq0}$ and TCFPP-II process $\{W_{\beta}^{f}(t)\}_{t\geq0}$, by subordinating it again with a subordinator and an inverse subordinator, respectively. As it clearly shown in \eqref{G-TCFPP-I} and \eqref{G-TCFPP-II}, the subordination of subordinator and inverse subordinator yields again a subordinator and an inverse subordinator, respectively. Hence, further subordination leads again to the processes of type TCFPP-I $\{Q_{\beta}^{f}(t)\}_{t\geq0}$ and TCFPP-II $\{W_{\beta}^{f}(t)\}_{t\geq0}$. This is also valid for $n$-iterated subordination. 
\end{remark}

\begin{theorem}
	The TCFPP-II $\{W_{\beta}^{f}(t)\}_{t\geq0}$ is a renewal process with {\it iid} waiting times $\{J_n\}_{n\geq1}$ with distribution
	\begin{equation}\label{renewal-1}
	\mathbb{P}[J_n>t]=\mathbb{E}[e^{-\lambda E_{\phi}(t)}],
	\end{equation}
	where $ E_{\phi}(t)$ is the inverse subordinator corresponding to $\phi(s)=(f(s))^\beta$.
	\begin{proof} Using \eqref{FPP-main} and Corollary \ref{lemma-ebeta}, we have 
		\begin{equation*}
		\{W_{\beta}^{f}(t)\}_{t\geq0}=\{N_{\beta}(E_f(t))\}_{t\geq0}\stackrel{d}{=}\{N(E_{\beta}(E_f(t)))\}_{t\geq0}\stackrel{d}{=}\{N(E_\phi(t))\}_{t\geq0},
		\end{equation*}
		where $\phi(s)=(f(s))^\beta.$  Therefore, the TCFPP-II $\{W_{\beta}^{f}(t)\}_{t\geq0}$ is a Poisson process time-changed by an inverse subordinator $\{E_\phi(t)\}_{t\geq0}$ corresponding to Bernstein function $\phi(s)=(f(s))^\beta.$ From   
		\cite[Theorem 4.1]{mnv}, we deduce that the time-changed Poisson process
		\begin{equation*}
		\{W_{\beta}^{f}(t)\}_{t\geq0}\stackrel{d}{=} \{N(E_\phi(t))\}_{t\geq0}  			
		\end{equation*}
		is a renewal process with {\it iid} waiting times $\{J_n\}_{n\geq1}$ having the distribution \eqref{renewal-1}.
	\end{proof}
\end{theorem}
\begin{remark}
	By \cite[Remark 5.4]{mnv}, the {\it pmf} $\eta_{\beta}^{f}(n|t,\lambda)$, given in \eqref{pmf-TCFPP-II}, of the TCFPP-II $\{W_{\beta}^{f}(t)\}_{t\geq0}$ satisfies
	\begin{equation*}
	\phi(\partial_t)\eta_{\beta}^{f}(n|t,\lambda)=-\lambda\left(\eta_{\beta}^{f}(n|t,\lambda)-\eta_{\beta}^{f}(n-1|t,\lambda)\right)+H(x)\psi_\phi (t,\infty)
	\end{equation*}
	in the mild sense, where $\phi(s)=\left(f(s)\right)^{\beta}$, $\psi_{\phi}(\cdot)$ is the L\'evy measure associated to Bernstein function $\phi(s)$ and $H(x)=I(x\geq0)$ is the Heaviside function.
\end{remark}

We next present the bivariate distributions of the TCFPP-II, which generalizes a result by \cite[Theorem 2.1]{Ors-Pol-int-FPP}. Let $F(t)$ be the distribution function of the waiting time $J_n$ and $S_n=J_1+\ldots+J_n$ be the time of $n$th jump. Since $J_n$'s are {\it iid}, we have that $\mathbb{P}[S_n\leq t]=F^{\ast n}(t)$, where $F^{\ast n}(t)$ denotes the $n$-fold convolution of $F(t)$. For $n,k\geq1$, define $\tau^{(k)}_n=S_{n+k}-S_n\stackrel{d}{=}S_k$, where $\tau^{(k)}_n$ is the time elapsed between $n$-th and $(n+k)$-th jump. Clearly, $\mathbb{P}[\tau^{(1)}_n\in dt]=dF(t)$ and $\mathbb{P}[\tau^{(k)}_n\in dt]=dF^{\ast k}(t)$, for $k\geq1$.
\begin{theorem}\label{bivariate-thm} Let $0\leq s<t$ and $0\leq m\leq n$ be nonnegative integers. Let $\{E_\phi(t)\}_{t\geq0}$ be the inverse subordinator corresponding to $\phi(s)=(f(s))^\beta.$
	The TCFPP-II  $\{W_{\beta}^{f}(t)\}_{t\geq0}$ has the bivariate distributions given by,
	\begin{align*}
	\mathbb{P}[W_\beta^f(s)&=m,W_\beta^f(t)=n]\\&=\left\{
	\begin{array}{ll}
	\displaystyle\int_{0}^{s}\mathbb{E}[e^{-\lambda E_{\phi}(t-u)}]dF^{\ast m}(u),~   \mbox{if } n=m\geq0, \\ &\\
	\displaystyle\int_{0}^{s}dF^{\ast m}(u) \int_{s-u}^{t-u}dF(v)\int_{0}^{t-(u+v)}\mathbb{E}[e^{-\lambda E_{\phi}(t-u-v-x)}]dF^{\ast n-m-1}(x), & \mbox{if }  n\geq m+1,
	\end{array}
	\right.
	\end{align*}
	where $F(t)=1-\mathbb{E}[e^{-\lambda E_{\phi}(t)}]$ and $dF^{\ast n}(t)$ is the $n$-fold convolution of $dF(t),~n\geq 1$, with $dF^{*0}(t)=\delta_0(t)$, the Dirac delta function at zero.
	\begin{proof}
		\noindent	{\it Case 1}: When $n=m$, we have (see Figure \ref{fig:mequaln}) 
		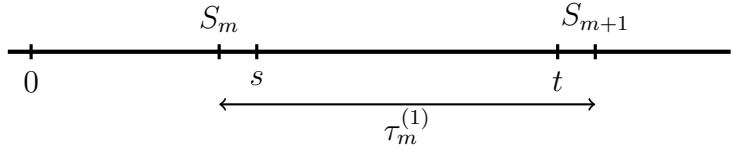
\begin{figure}[ht!]
			\begin{tikzpicture}[xscale=1]
			\draw [ultra thick] (-.31,0) -- (9.3,0);
			\draw [very thick](0,-.1) -- (0, .1) node[below] at (0,-.1)%
			{0};
			\draw [very thick](2.5,-.1) -- (2.5, .1) node[above] at (2.5,.1) {$S_m$};
			\draw [very thick](3,-.1) -- (3, .1) node[below] at (3,-.1)%
			{$s$};
			\draw [very thick](7,-.1) -- (7, .1) node[below] at (7,-.1) {$t$};
			\draw [very thick] (7.5,-.1) -- (7.5, .1)node[above] at (7.5,.1) {$S_{m+1}$};
			\draw[-][thick,<->] (2.5,-0.7) -- (7.5,-0.7) node[below] at (5,-.6) {$\tau_m^{(1)}$};
			\end{tikzpicture}
			\caption{Waiting times between events for $m=n$}
			\label{fig:mequaln}
		\end{figure}\\
		\begin{align*}
		\mathbb{P}[W_\beta^f(s)=m,W_\beta^f(t)=m]&=\mathbb{P}\big[0<S_m\leq s;~t<S_{m+1}\big]=\mathbb{P}\big[0<S_m\leq s;~t<S_{m}+\tau_m^{(1)}\big]\nonumber\\
		&=\mathbb{P}\big[0<S_m\leq s;~\tau_m^{(1)}>t-S_m\big]\\
		&=\int_{0}^{s}dF^{*m}(u)\mathbb{P}[\tau_m^{(1)}>t-u\big]~~(\text{since }S_m\text{ and }\tau_m^{(1)} \text{ are independent})\\
		&=\int_{0}^{s}\mathbb{E}[e^{-\lambda E_{\phi}(t-u)}]dF^{\ast m}(u).
		\end{align*}
		
		\noindent{\it Case 2:} When $n\geq m+1$, it follows that (see Figure \ref{fig:mlessn})
		\begin{figure}[ht!]
			\begin{tikzpicture}[xscale=1.5]
			\draw [ultra thick] (-.31,0) -- (9.3,0);
			\draw [very thick](0,-.1) -- (0, .1) node[below] at (0,-.1)%
			{0};
			\draw [very thick](2.5,-.1) -- (2.5, .1) node[above] at (2.5,.1) {$S_m$};
			\draw [very thick](3,-.1) -- (3, .1) node[below] at (3,-.1)%
			{$s$};
			\draw [very thick](3.7,-.1) -- (3.7, .1)node[above] at (3.7,.1) {$S_{m+1}$};
			\draw [very thick](6.3,-.1) -- (6.3, .1)node[above] at (6.3,.1) {$S_n$};
			\draw [very thick](7,-.1) -- (7, .1) node[below] at (7,-.1) {$t$};
			\draw [very thick] (8,-.1) -- (8, .1)node[above] at (8,.1) {$S_{n+1}$};
			
			\draw[-] [thick,<->] (2.5,-1) -- (3.7,-1) node[below] at (3.1,-1.1) {$\tau^{(1)}_m$};
			\draw[-][thick,<->] (3.7,-0.5) -- (6.3,-0.5) node[below] at (5,-.6) {$\tau_{m+1}^{(n-m-1)}$};
			\draw[-][thick,<->] (6.3,-1) -- (8,-1) node[below] at (7.2,-1.1) {$\tau^{(1)}_n$};
			\end{tikzpicture}
			\caption{Waiting times between events for $n\geq m+1$}
			\label{fig:mlessn}
		\end{figure}
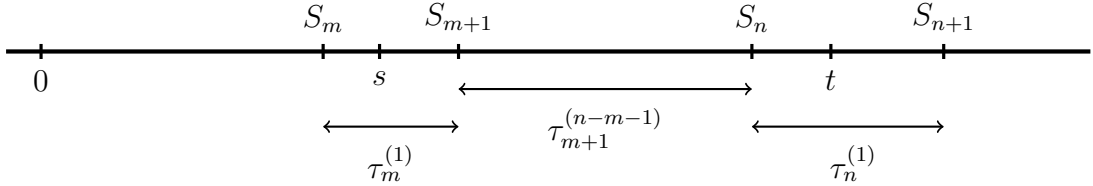
		\begin{align}
		\mathbb{P}[W_\beta^f(s)&=m,W_\beta^f(t)=n]\nonumber\\&=\mathbb{P}\big[0<S_m\leq s;~\tau^{(1)}_m>s-S_m;~\tau^{(1)}_m<t-S_m;0<\tau_{m+1}^{(n-m-1)}<t-S_{m+1};~\tau^{(1)}_m>t-S_n\big]\nonumber\\
		&=\mathbb{P}\big[0<S_m\leq s;s-S_m<\tau^{(1)}_m<t-S_m;0<\tau_{m+1}^{(n-m-1)}<t-S_{m}-\tau^{(1)}_m;\nonumber\\&\hspace*{10cm}\tau^{(1)}_n>t-S_m-\tau^{(1)}_m-\tau_{m+1}^{(n-m-1)}\big].\nonumber
		\end{align}
		Since the waiting times between events are {\it iid}, we have that
		\begin{align*}
		\mathbb{P}[W_\beta^f(s)&=m, ~W_\beta^f(t)=n]\\&
		=\int_{0}^{s}\mathbb{P}[S_m\in du] \int_{s-u}^{t-u}\mathbb{P}[\tau^{(1)}_m\in dv]\int_{0}^{t-(u+v)}\mathbb{P}[\tau_{m+1}^{(n-m-1)}\in dw]\nonumber\int_{t-(u+v+w)}^{\infty}\mathbb{P}\left[\tau^{(1)}_n\in dx\right]\nonumber\\
		&=\int_{0}^{s}dF^{\ast m}(u) \int_{s-u}^{t-u}dF(v)\int_{0}^{t-(u+v)}dF^{\ast (n-m-1)}(w)\mathbb{P}[\tau^{(1)}_n>(t-u-v-w)]\nonumber\\
		&=\int_{0}^{s}dF^{\ast m}(u) \int_{s-u}^{t-u}dF(v)\int_{0}^{t-(u+v)}\mathbb{E}[e^{-\lambda E_{\phi}(t-u-v-w)}]dF^{\ast (n-m-1)}(w)\nonumber,
		\end{align*}
		which completes the proof. 
	\end{proof}
\end{theorem}
\noindent Let us examine a special case of Theorem \ref{bivariate-thm} for the FPP.
\begin{remark}
	It is known (see \cite{mnv}) that the FPP $\{N_\beta(t)\}_{t\geq0}=\{N(E_\beta(t))\}_{t\geq0}$ is a renewal process whose inter-arrival times follow the Mittag-Leffler distribution, that is, 
	\begin{equation*}
	\mathbb{P}[J_n\leq t]=F(t)=1-L_\beta(-\lambda t^\beta),~0<\beta<1,
	\end{equation*}
	where $L_{\beta}(z)$ is the Mittag-Leffler function defined in \eqref{Mittag-Leffler-function}. Let us define $L_{\alpha,0}^0(-\lambda t^\beta):=t\delta_0(t)$, where $\delta_0(t)$ is the Dirac delta function at zero. This implies for $t\geq0$ (see \cite{Ors-Pol-int-FPP} and references therein),
	\begin{equation}\label{remark-fpp}
	\begin{aligned}
	\mathbb{P}[S_m\in dt]&=\mathbb{P}[\tau^{(m)}_n\in dt]=dF^{\ast m}(t)=\lambda^m t^{m\beta-1}L_{\beta,m\beta}^{m}(-\lambda t^\beta)dt,~m\geq0,
	\end{aligned} 
	\end{equation}
	where $L_{\alpha,\beta}^\gamma(z)$ is the generalized Mittag-Leffler function defined in \eqref{Mittag-Leffler-general}. The LT of the inverse $\beta$-stable subordinator is given by (see \cite[eq. (16)]{bingham71})
	\begin{equation}\label{LT-E-beta}
	\mathbb{E}[e^{-\lambda E_\beta(t)}]=L_\beta(-\lambda t^\beta).
	\end{equation}  Using  \eqref{remark-fpp}, \eqref{LT-E-beta} and Theorem \ref{bivariate-thm}, the bivariate distribution of the FPP, when $n=m\geq0$, is
	\begin{equation*}
	\mathbb{P}[N_\beta(s)=m,N_\beta(t)=m]=\lambda^m\int_{0}^{s}u^{m\beta-1}L^m_{\beta,m\beta}(-\lambda u^\beta)L_\beta(-\lambda(t-u)^\beta)du,~ m\geq 0.
	\end{equation*}
	For $n\geq m+1$,
	\begin{align*}
	\mathbb{P}[N_\beta(s)=m,N_\beta(t)=n]&
	=\lambda^n\int_{0}^{s} u^{m\beta-1}L_{\beta,m\beta}^{m}(-\lambda u^\beta)\int_{s-u}^{t-u} v^{\beta-1}L^1_{\beta,\beta}(-\lambda v^\beta)\int_{0}^{t-(u+v)} x^{\beta(n-m-1)-1}\\ 
	&~~~\times L_{\beta,\beta(n-m-1)}^{n-m-1 }(-\lambda x^\beta) L_{\beta}(-\lambda(t-u-v-x)^\beta)]dudvdx,
	\end{align*} 
	which coincides (2.9) of \cite{Ors-Pol-int-FPP}. Indeed, it is shown in \cite[eq. (2.6)]{Ors-Pol-int-FPP} that
	\begin{align*}
	\mathbb{P}[N_\beta(s)=m,N_\beta(t)=n]&=\lambda^n\displaystyle\int_{0}^{s}u^{m\beta -1}L_{\beta,\beta m}^m(-\lambda u^\beta)du \displaystyle\int_{s-u}^{t-u}v^{\beta-1}L_{\beta,\beta}^1(-\lambda v^\beta)\\&~~\times(t-u-v)^{\beta(n-m-1)}L^{n-m}_{\beta,\beta(n-m-1)+1}(-\lambda(t-u-v)^\beta)dv,~n\geq m+1.
	\end{align*}
	When $\beta=1$, $L_{1,1}^1(x)=e^x$, and $L^1_{1,m}(x)=e^x/(m-1)!$ and
	\begin{align*}
	\mathbb{P}[N(s)=m,N(t)=n]&=\left\{
	\begin{array}{ll}
	\dfrac{\lambda^ms^m}{m!}e^{-\lambda t},  & \mbox{if } n=m, \\ &\\
	\dfrac{\lambda^ns^m(t-s)^{n-m}}{n!}\displaystyle\binom{n}{m}e^{-\lambda t}, & \mbox{if }  n\geq m+1,
	\end{array}
	\right.
	\end{align*}
	the bivariate distribution of the Poisson process, as expected.
\end{remark}


\section{Simulation}\label{sec:simulation}
\noindent In this section, we present  simulated sample paths for some TCFPP-I and TCFPP-II processes. The sample paths for the FNBP, the FPP subordinated with tempered $\alpha$-stable subordinator (FPP-TSS) and the FPP subordinated with inverse Gaussian subordinator (FPP-IGN) are presented for a chosen set of parameters. The simulations of the corresponding TCFPP-II process of the FPP subordinated with inverse gamma subordinator (FPP-IG), the FPP subordinated with inverse tempered $\alpha$-stable subordinator (FPP-ITSS), and the FPP subordinated with inverse of inverse Gaussian subordinator (FPP-IIGN) are also given in this section.   We first present the algorithm for simulation of the FPP.

\begin{algorithm}[\bf Simulation of the FPP]\label{simulation-fpp}
	This algorithm (see \cite{Cahoy2010}) gives the number of events $N_\beta(t),~0<\beta<1$ of the FPP up to a fixed time $T$. 
	\begin{enumerate}[(a)]
		\item Fix the parameters $\lambda>0$ and $0<\beta<1$ for the FPP.
		\item	Set $n=0$ and $t=0.$
		\item Repeat while $t<T$
		\begin{enumerate}
			\item[] Generate three independent uniform random variables $U_i\sim U(0,1)$, $i = 1, 2, 3$.
			\item[]  Compute (see \cite{Kanter1975}) $$dt=\frac{|\ln U_1|^{1/\beta}}{\lambda^{1/\beta}}\frac{\sin(\beta\pi U_2)[\sin(1-\beta)\pi U_2]^{1/\beta-1}}{[\sin(\pi U_2)]^{1/\beta}|\ln U_3|^{1/\beta-1}}.$$
			\item[]  $t=t+dt$ and $n=n+1$.
			
		\end{enumerate}
		\item Next $t$.
	\end{enumerate}
	Then $n$ denotes the number of events $N_\beta(t)$ occurred up to time $T$.
\end{algorithm} 
\noindent We next present the algorithms for the simulation of the gamma subordinator, the tempered $\alpha$-stable subordinator and the inverse Gaussian subordinator. The generated sample paths from these algorithms will then be used to simulate the inverse subordinator and the TCFPP-I.
\begin{algorithm}[\bf Simulation for the gamma subordinator]\label{simulation-gamma}
	\begin{enumerate}[(a)]
		\item[]	
		\item Fix the parameters $\alpha$ and $p$ for gamma subordinator.
		\item  Choose an interval $[0, T].$ Choose $n+1$ uniformly spaced time points $0=t_0,t_1,\ldots,t_n=T$ with $h=t_2-t_1.$
		\item Simulate $n$ independent gamma random variables $Q_i\sim G(\alpha,ph),1\leq i\leq n$, using GSS algorithm (see \cite[p. 321]{gamma-simu}).
		
		\item The discretized sample path of $Y(t)$ at $t_i$ is $Y(ih)=Y(t_i)=\sum_{j=1}^{i}Q_{j}, 1\leq i\leq n$ with $Q_0=0.$ 
		
	\end{enumerate}
\end{algorithm}

\begin{algorithm}[\bf  Simulation for the TSS]\label{simulation-tss}
	
	\begin{enumerate}[(a)]
		\item[]	
		\item 	Choose the  parameters $\mu>0$ and $0<\alpha<1$.
		\item Choose an interval $[0, T].$ Choose $n+1$ time points $0=t_0,t_1,\ldots,t_n=T.$
		\item Simulate $D_\alpha^{\mu (t_i-t_{i-1})^{1/\alpha}}(1)$ for $1\leq i\leq n$ from the Algorithm 3.2 of \cite{Hofert2011}. 
		\item Compute the increments
		\begin{equation*}
		\Delta D_\alpha^{\mu{(i)}}=D^{\mu}_\alpha(t_i)-D^{\mu}_\alpha(t_{i-1})=(t_i-t_{i-1})^{1/\alpha}D_\alpha^{\mu (t_i-t_{i-1})^{1/\alpha}}(1),~ 1\leq i\leq n,
		\end{equation*}
		with $D^{\mu}_\alpha(0)=0.$
		\item The discretized sample path of $D^\mu_\alpha(t)$ at $t_i$ is $D^\mu_\alpha(t_i)=\sum_{j=1}^{i}\Delta D_\alpha^{\mu{(j)}}, 1\leq i\leq n.$		
	\end{enumerate}
\end{algorithm}

\begin{algorithm}[\bf Simulation of the IGN subordinator]\label{simulation-ig}
	The algorithm to generate the IGN random variables is given in  \cite[p. 183]{ContTan2004}.
	\begin{enumerate}[(a)]
		\item  Choose an interval $[0, T].$ Choose $n+1$ uniformly spaced time points $0=t_0,t_1,\ldots,t_n=T$ with $h=t_2-t_1.$
		\item  Since IGN subordinator $\{G(t)\}_{t\geq0}$ has independent and stationary increments, $F_i\equiv G(t_i)-G(t_{i-1})\stackrel{d}= G(h)\sim$ IGN$(h, 1)$ for $1\leq i\leq n$ and $h=T/n$. Now generate $n$ {\it iid} IGN variables $F_i$'s as follows (see \cite[p. 183]{ContTan2004}, therein substituted $\delta=1=\gamma$):
		\begin{enumerate}
			\item[]  Generate a standard normal random variable $N$.
			\item[]  Assign $X=N^2$.
			\item[]  Assign $Y=h+\frac{X}{2}-\frac{1}{2}\sqrt{4hX+X^2}$.
			\item[]  Generate a uniform $[0,1]$ random variable $U$.
			\item[]  If $U\leq \frac{h}{h+Y}$, return $Y$; else return $\frac{h^2}{Y}$.
		\end{enumerate}
		\item Assign $G(t_0) =0$. The discretized sample path of $G(t)$ at $t_i$ is $G(t_i)=\sum_{j=1}^{i}F_j, 1\leq i\leq n.$
		
	\end{enumerate}
\end{algorithm}
\noindent Consider next the algorithm to simulate the inverse subordinator $\{E_f(t)\}_{t\geq 0}$. We first define $E_f^\delta(t)$ with the step length $\delta$ as (see \cite{Arun-inverse-gamma})
\begin{equation*}
E_f^\delta(t) = (\min\{n\in\mathbb{N}:D_f(\delta n) > t\} -1)\delta, ~~n = 1,2,...,
\end{equation*}
where $D_f(\delta n)$ is the value of the subordinator $D_f(t)$ evaluated at $\delta nδ$, which can be simulated by using the method presented above. Observe that trajectory of $E_f^\delta (t)$ has increments
of length $\delta$ at random time instants governed by process $D_f(t)$ and therefore $E_f^\delta(t)$ is the approximation of operational time. 
\begin{algorithm}[\bf Simulation of the inverse subordinator]
	
	\begin{enumerate}[(a)]
		\item[]		
		\item Fix the parameters for the inverse subordinator, whichever under consideration.
		\item Choose $n$ uniformly spaced time points $0=t_1,t_2,\ldots,t_n=T$ with $h=t_2-t_1.$
		\item Let $i=0$  and $t=0$.
		\item Repeat while $t<T$
		\begin{enumerate}
			\item[] Generate an independent $D_f(t)$ random variables with
			$Q_i\sim D_f(h).$
			\item[]Set	$W(\lceil t/h\rceil+1):=h*i,\ldots,W(\lfloor (t+Q_i)/h\rfloor+1):=h*i$.
			\item[] $i=i+1,~t=t+Q_i.$
			\item[] Next $t$.
		\end{enumerate}
		\item The discretized sample path of $E_f(t)$ at $t_i$ is  $W_{i}, 1\leq i\leq n$ with $W_0=0.$ 
		
	\end{enumerate}
\end{algorithm}
\noindent Note that the simulations for the inverse of gamma subordinator, the inverse of tempered $\alpha$-stable subordinator and the inverse of inverse Gaussian subordinator can be done using the above algorithm by replacing the special case for the subordinator.\\
\noindent We next present a general algorithm to simulate the TCFPP-I, namely the FNBP, the FPP-TSS and the FPP-IGN processes. The same algorithm can be used to simulate the TCFPP-II, namely the FPP-IG, the FPP-ITSS and the FPP-IIGN processes.
\begin{algorithm}[\bf  Simulation of the TCFPP-I and the TCFPP-II]\label{simulation-fnbp}
	\begin{enumerate}[(a)]
		\item[]
		
		\item Fix the parameters for the subordinator (inverse subordinator), under consideration. Choose the fractional index $\beta ~(0<\beta<1)$ and rate parameter $\lambda>0$ for the FPP.
		\item  Fix the time $T$ for the time interval $[0,T]$ and choose $n+1$ uniformly spaced time points $0=t_0,t_1,\ldots,t_n=T$ with $h=t_2-t_1$.
		\item Simulate the values $X(t_i),1\leq i\leq n,$ of the subordinator (inverse subordinator) at $t_1,\ldots t_n,$ using the algorithm for respective subordinator (inverse subordinator).
		\item  Using the values $X(t_i),1\leq i\leq n,$ generated in Step (c), as time points, compute the number of events of the FPP $N_\beta(X(t_i)),1\leq i\leq n,$ using Algorithm \ref{simulation-fpp}.
	\end{enumerate}
\end{algorithm}

\begin{figure}[H]
	\begin{subfigure}{0.5\textwidth}
		\caption{Sample paths of the FNBP for $\beta=0.6,\alpha=3.0,p=4.0$ and $\lambda=1.5.$}
		\includegraphics[width=0.8\linewidth]{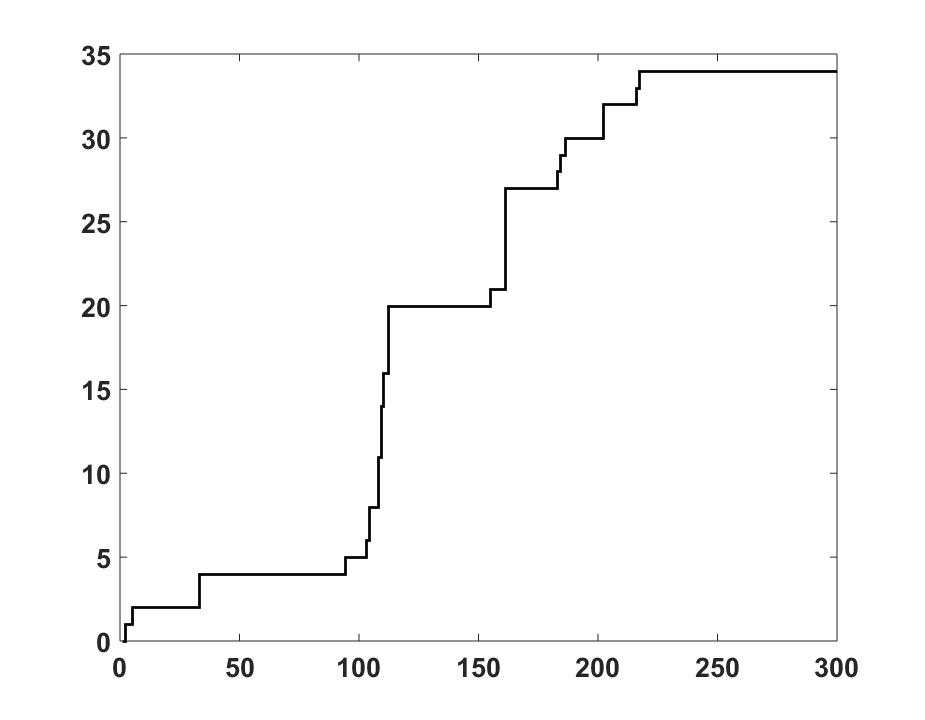}	\end{subfigure}~~
	\begin{subfigure}{0.5\textwidth}
		\caption{Sample paths of the FNBP for $\beta=0.90,\alpha=3.0,p=4.0$ and $\lambda=2.0.$}
		\includegraphics[width=0.8\linewidth]{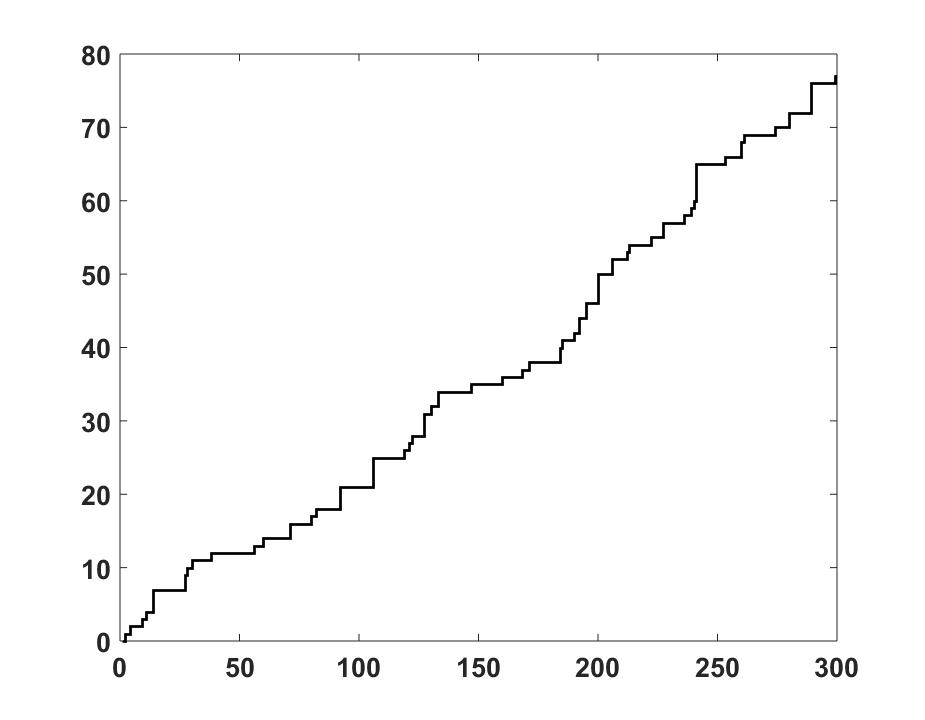}
	\end{subfigure}
	\caption{Sample paths of the FNBP process}
\end{figure}

\begin{figure}[htb]
	\begin{subfigure}{0.5\textwidth}
		\caption{Sample paths of the FPP-IG for $\beta=0.6,\alpha=3.0,p=4.0$ and $\lambda=1.5.$}			\includegraphics[width=0.8\linewidth]{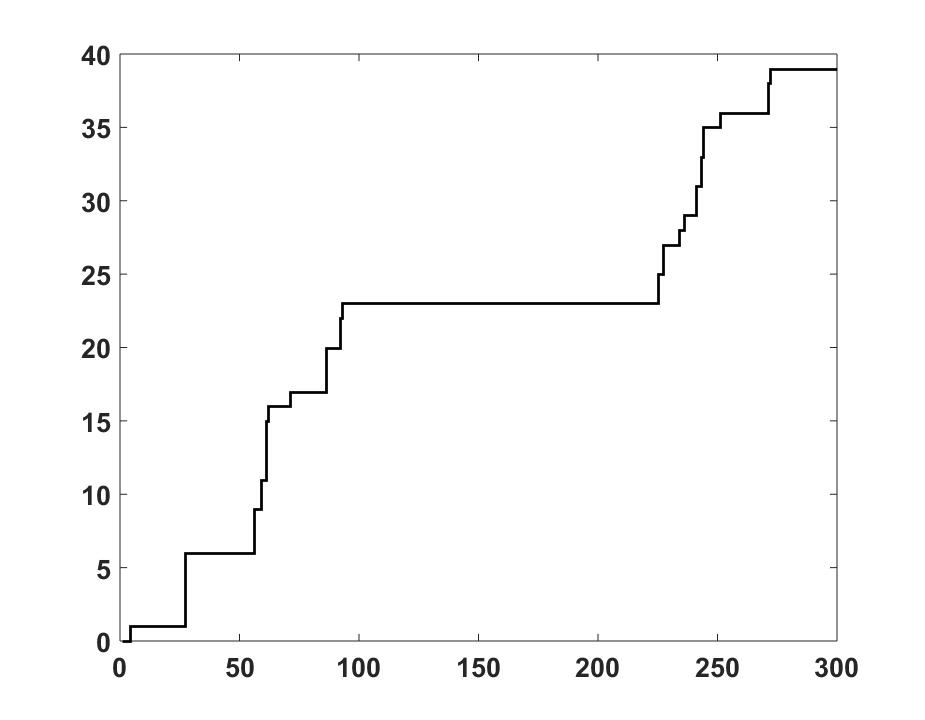}	\end{subfigure}~~
	\begin{subfigure}{0.5\textwidth}
		\caption{Sample paths of the FPP-IG for $\beta=0.90,\alpha=3.0,p=4.0$ and $\lambda=2.0.$}
		\includegraphics[width=0.8\linewidth]{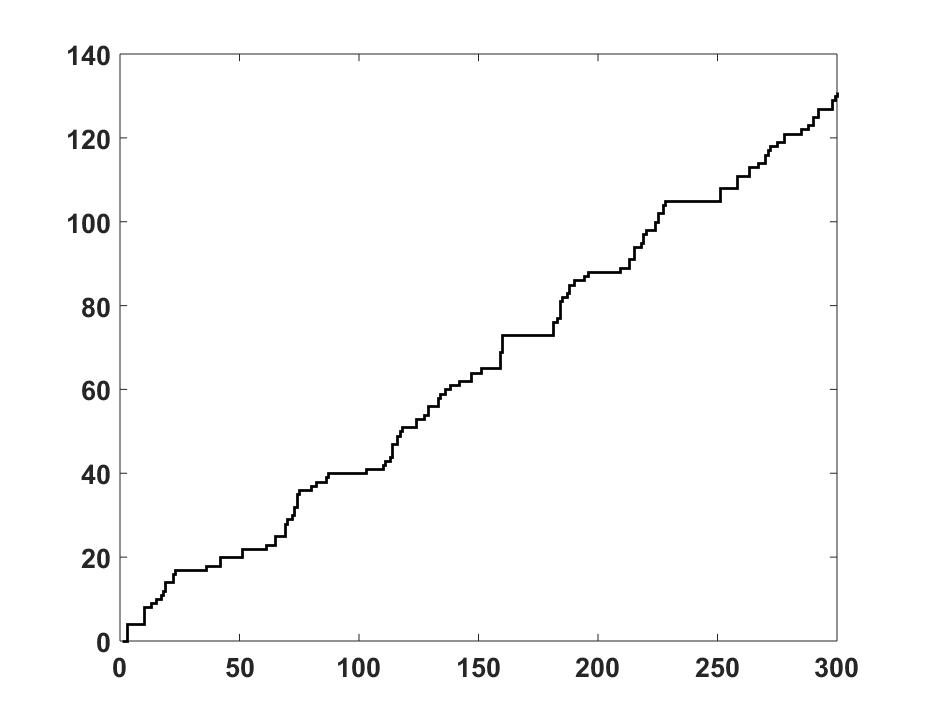}
	\end{subfigure}
	\caption{Sample paths of the FPP-IG process}
\end{figure}

\begin{figure}[htb]
	\begin{subfigure}{0.5\textwidth}
		\caption{Sample paths of the FPP-TSS process for $\beta=0.6,\mu=2.0,\alpha=0.5$ and $\lambda=1.5.$}
		\includegraphics[width=0.8\textwidth]{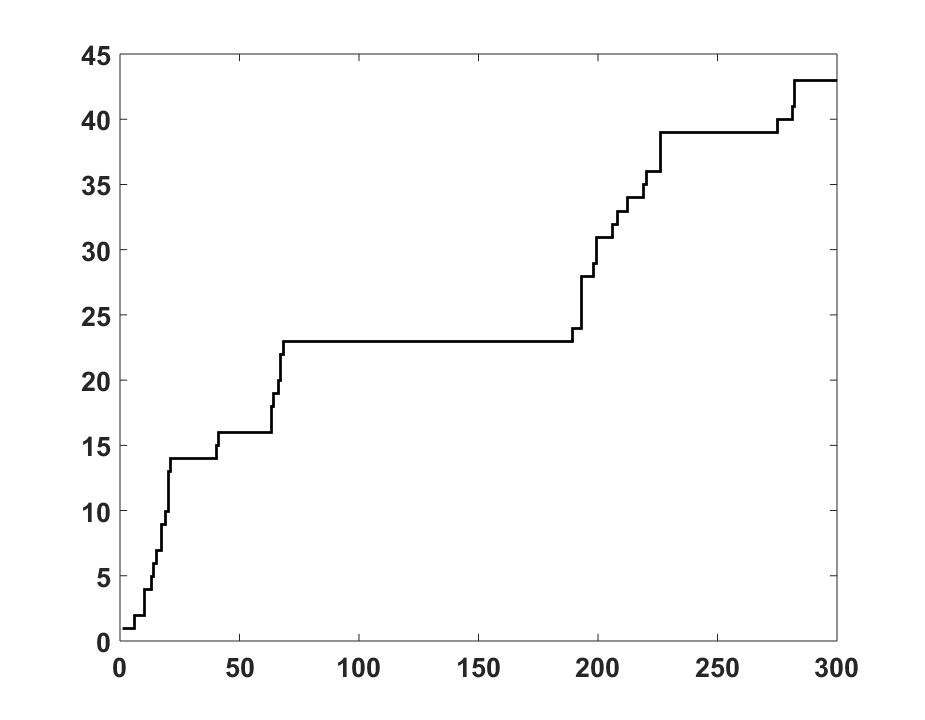}
	\end{subfigure}~
	\begin{subfigure}{0.5\textwidth}
		\caption{Sample paths of the FPP-TSS process for $\beta=0.9,\mu=2.0,\alpha=0.7$ and $\lambda=2.0.$}
		\includegraphics[width=0.8\textwidth]{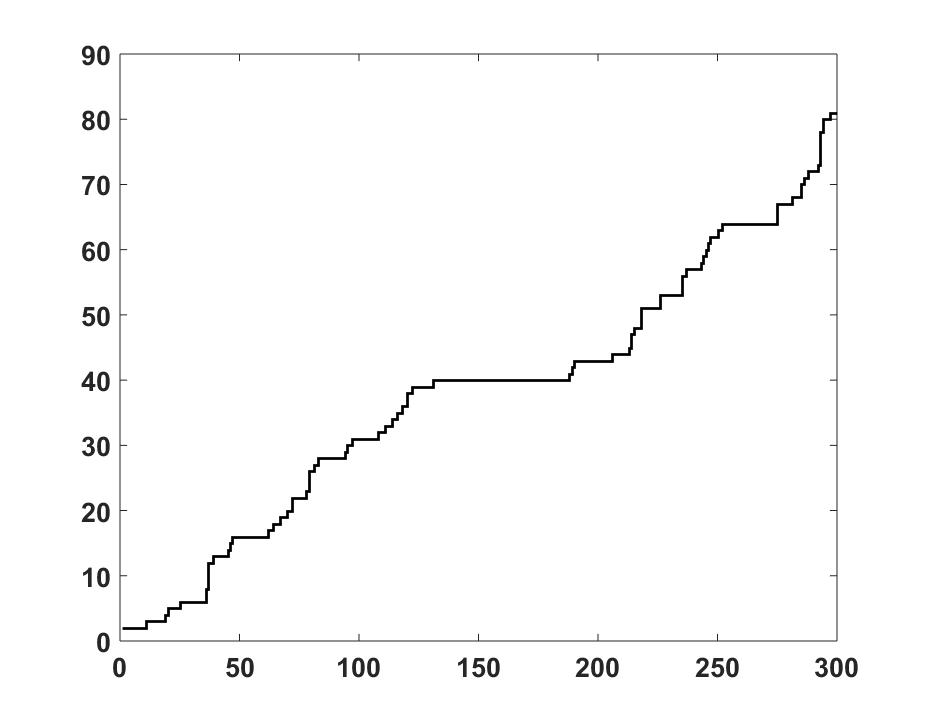}
	\end{subfigure}
	\caption{Sample paths of the FPP-TSS process}
\end{figure}

\begin{figure}[htb]
	\begin{subfigure}{0.5\textwidth}
		\caption{Sample paths of the FPP-ITSS process for $\beta=0.6,\mu=2.0,\alpha=0.5$ and $\lambda=1.5.$}
		\includegraphics[width=0.8\textwidth]{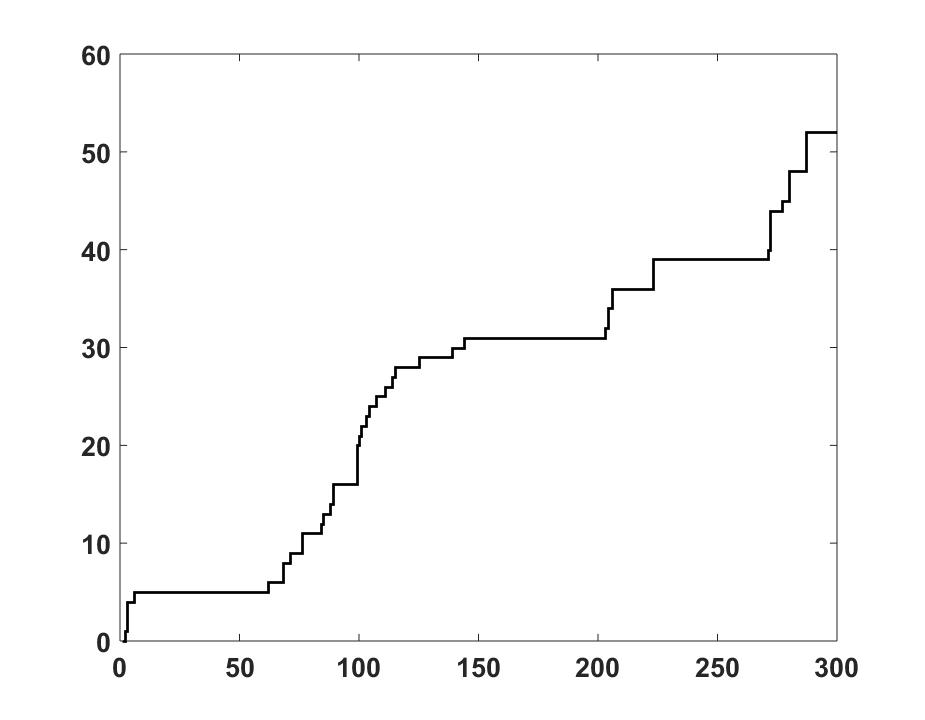}
	\end{subfigure}~
	\begin{subfigure}{0.5\textwidth}
		\caption{Sample paths of the FPP-ITSS process for $\beta=0.9,\mu=2.0,\alpha=0.7$ and $\lambda=2.0.$}
		\includegraphics[width=0.8\textwidth]{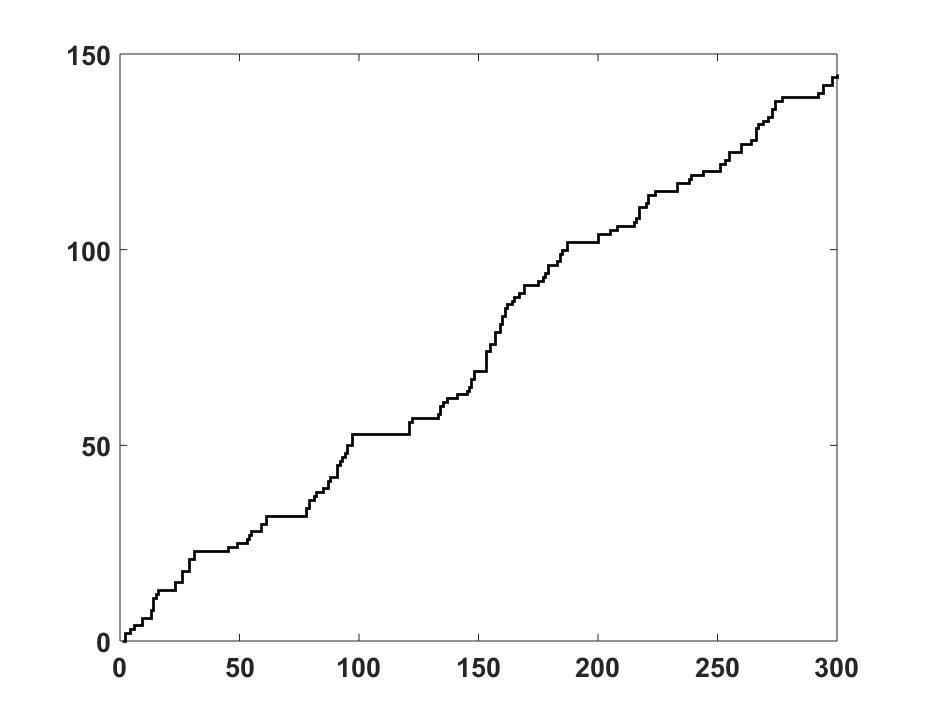}
	\end{subfigure}
	\caption{Sample paths of the FPP-ITSS process}
\end{figure}
\begin{figure}[htb]
	\begin{subfigure}{0.51\textwidth}
		\caption{Sample paths of the FPP-IGN process for $\beta=0.6,\delta=1=\gamma$ and $\lambda=1.5.$}
		\includegraphics[width=0.8\textwidth]{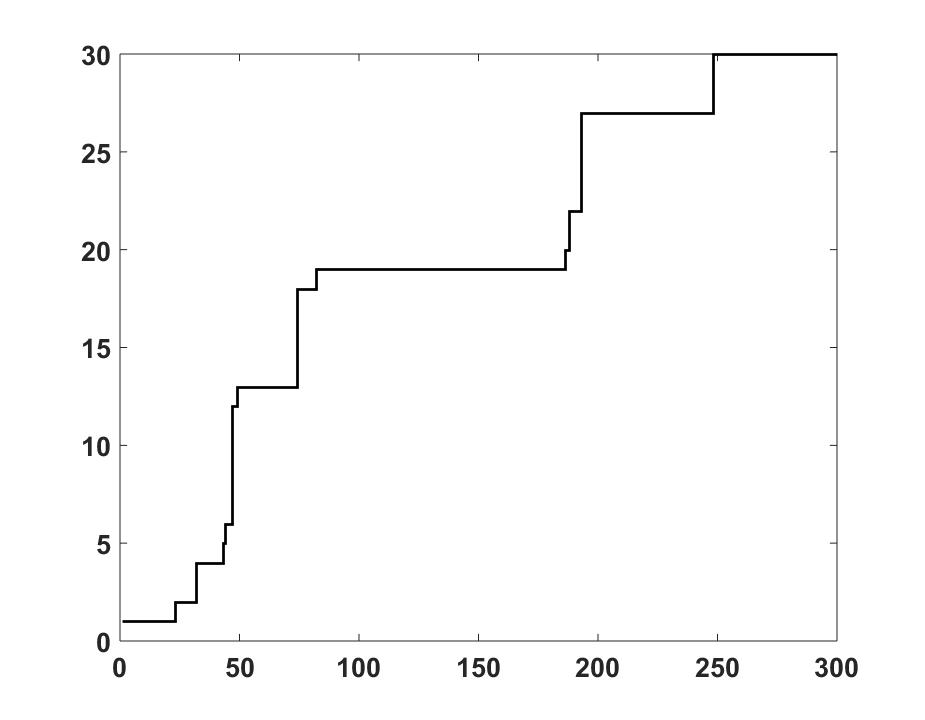}
	\end{subfigure}~
	\begin{subfigure}{0.51\textwidth}
		\caption{Sample paths of the FPP-IGN process for $\beta=0.9,\delta=1=\gamma$ and $\lambda=2.0.$}
		\includegraphics[width=0.8\textwidth]{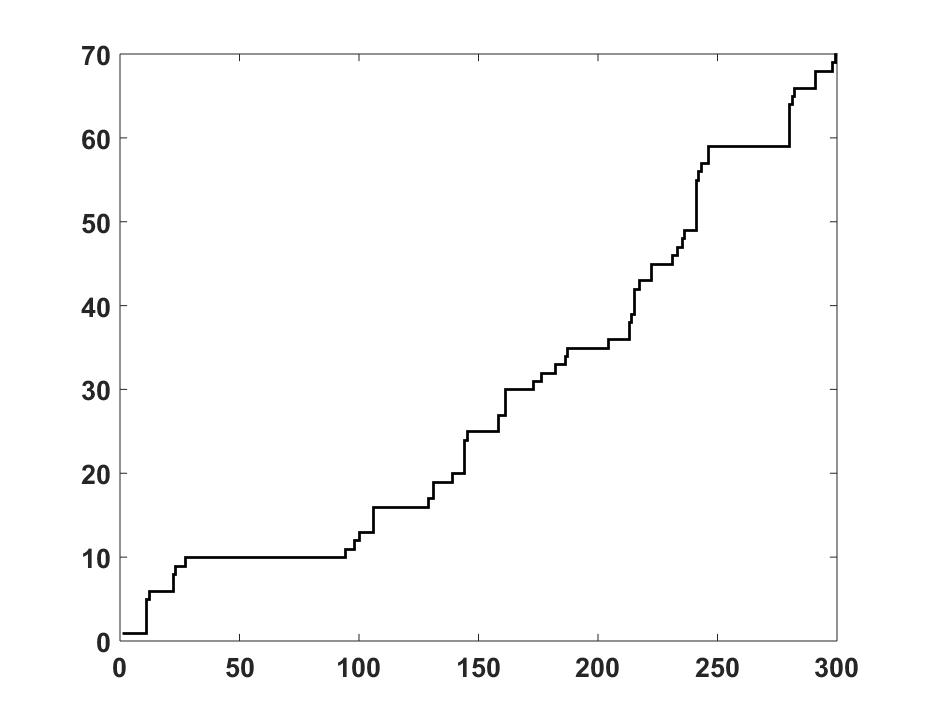}
	\end{subfigure}
	\caption{Sample paths of the FPP-IGN process}
\end{figure}

\begin{figure}[H]
	\begin{subfigure}{0.51\textwidth}
		\caption{Sample paths of the FPP-IIGN process for $\beta=0.6,\delta=1=\gamma$ and $\lambda=1.5.$}
		\includegraphics[width=0.8\textwidth]{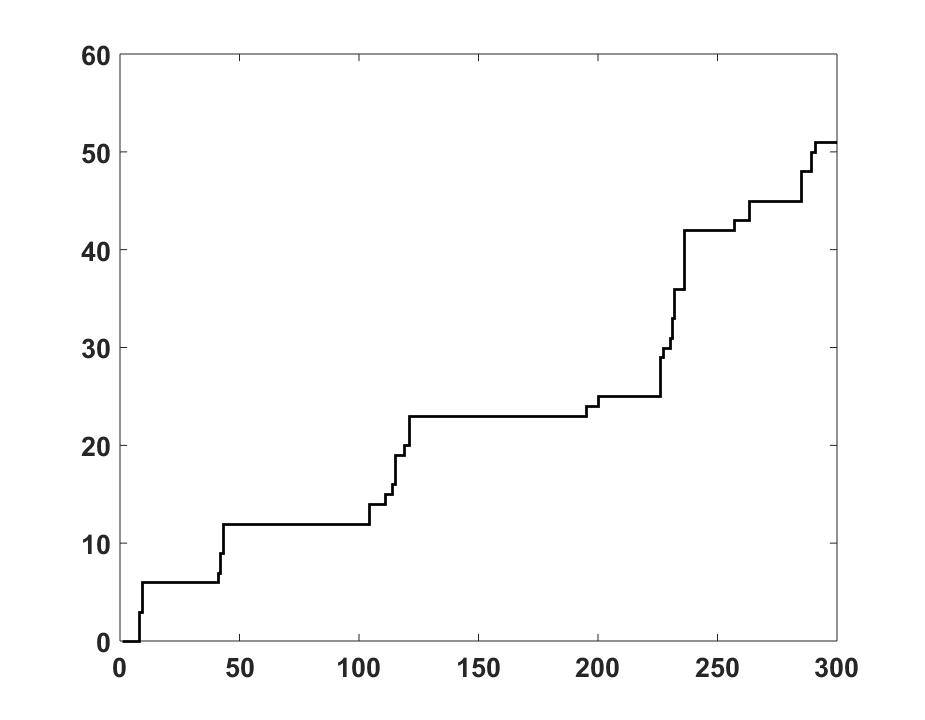}
	\end{subfigure}~
	\begin{subfigure}{0.51\textwidth}
		\caption{Sample paths of the FPP-IIGN process for $\beta=0.9,\delta=1=\gamma$ and $\lambda=2.0.$}
		\includegraphics[width=0.8\textwidth]{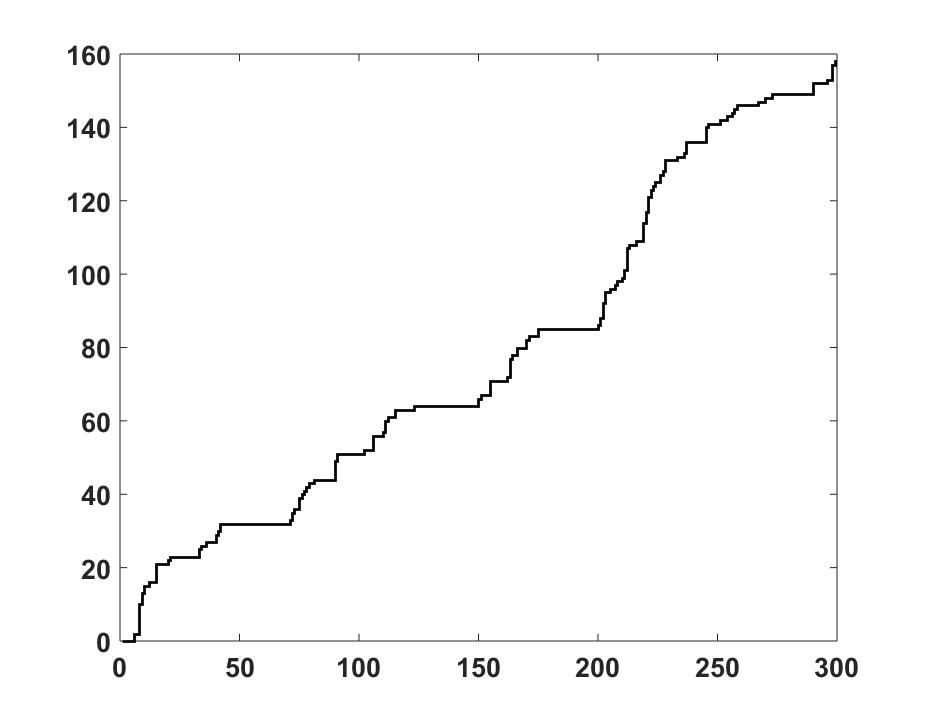}
	\end{subfigure}
	\caption{Sample paths of the FPP-IIGN process}
\end{figure}
\FloatBarrier
\section*{Acknowledgments}
A part of this work was done while the second author was
visiting the Department of Statistics and Probability, Michigan State University, during Summer-2016.

\def\cprime{$'$}
\providecommand{\bysame}{\leavevmode\hbox to3em{\hrulefill}\thinspace}
\providecommand{\MR}{\relax\ifhmode\unskip\space\fi MR }
\providecommand{\MRhref}[2]{%
	\href{http://www.ams.org/mathscinet-getitem?mr=#1}{#2}
}
\providecommand{\href}[2]{#2}

\end{document}